\documentclass[]{interact}

\theoremstyle{plain}
\newtheorem{theorem}{Theorem}[section]

\theoremstyle{remark}
\newtheorem{remark}{Remark}

\usepackage{natbib}
\bibpunct[, ]{(}{)}{;}{a}{}{,}

\usepackage{nomencl}

\usepackage{mathtools}
\usepackage{times}
\usepackage{amsmath,amssymb}
\usepackage{amsthm}
\usepackage{graphicx}

\usepackage[inline]{enumitem}
\usepackage{array}
\usepackage{multicol}
\usepackage{multirow}
\usepackage{tabulary}
\usepackage{longtable}
\usepackage{booktabs}
\usepackage{caption}

\usepackage{algorithm}
\usepackage{algpseudocode}

\usepackage{tcolorbox}
\usepackage{mathtools}
\usepackage{tikz}
\usetikzlibrary{arrows}
\usetikzlibrary{calc}
\usetikzlibrary{arrows,decorations.markings}
\usepackage{nomencl}
\usepackage{tikz}
\usetikzlibrary{arrows}
\usetikzlibrary{calc}
\usetikzlibrary{arrows,decorations.markings}

\algnewcommand{\IIf}[1]{\State\algorithmicif\ #1\ \algorithmicthen}
\algnewcommand{\EndIIf}{\unskip\ \algorithmicend\ \algorithmicif}

\DeclareMathOperator{\diag}{diag}

\DeclareMathOperator*{\mini}{min.}
\DeclareMathOperator*{\lexmin}{lexmin.}

\definecolor{light-gray}{gray}{0.95}
\definecolor{dark-gray}{gray}{0.5}
\definecolor{mygray}{gray}{0.75}
\newcommand{\BIN}{\begin{bmatrix}}
\newcommand{\BOUT}{\end{bmatrix}}

\newcommand{\umA}{\underline{\mathcal{A}}}

\newcommand{\mI}{{\mathcal{I}}}
\newcommand{\bE}{{\mathbb{E}}}
\newcommand{\bI}{{\mathbb{I}}}
\newcommand{\mA}{{\mathcal{A}}}

\definecolor{orange}{rgb}{0.99,0.69,0.07}
\definecolor{lightgray}{gray}{0.85}
\definecolor{light-gray}{gray}{0.95}
\definecolor{dark-gray}{gray}{0.5}

 \makeatletter
 \newcommand\fs@spaceruled{\def\@fs@cfont{\bfseries}\let\@fs@capt\floatc@ruled
   \def\@fs@pre{\vspace{5pt}\hrule height.8pt depth0pt \kern2pt}%
   \def\@fs@post{\kern2pt\hrule\relax}%
   \def\@fs@mid{\kern2pt\hrule\kern2pt}%
   \let\@fs@iftopcapt\iftrue}
 \makeatother

\newcommand{\parm}{\mathord{\color{black!33}\bullet}}%

\makenomenclature
\title{Sequential Hierarchical Least-Squares Programming for Prioritized Non-Linear Optimal Control}

\begin{document}

\articletype{}

\author{
\name{Kai Pfeiffer\textsuperscript{a}\thanks{CONTACT Kai Pfeiffer. Email: kaipfeifferrobotics@gmail.com} and Abderrahmane Kheddar\textsuperscript{b,c}}
\affil{
\textsuperscript{a}School of Mechanical and Aerospace Engineering, {Nanyang Technological University}, {{639798 Singapore}; 
\textsuperscript{b}Joint Robotics Laboratory (JRL)}, {UMI3218/RL}, {Tsukuba 305-8560, Japan}; 
\textsuperscript{c}Interactive Digital Human, University of Montpellier, CNRS,	LIRMM, UMR5506, 34095 Montpellier; 
}
}

\maketitle

\begin{abstract}
We present a sequential hierarchical least-squares programming solver with trust-region and hierarchical step-filter with application to prioritized discrete non-linear optimal control. It is based on a hierarchical step-filter which resolves each priority level of a non-linear hierarchical least-squares programming via a globally convergent sequential quadratic programming step-filter. Leveraging a condition on the trust-region or the filter initialization, our hierarchical step-filter maintains this global convergence property. 
The hierarchical least-squares programming sub-problems are solved via a sparse reduced Hessian based interior point method. It leverages an efficient implementation of the turnback algorithm for the computation of nullspace bases for banded matrices. We propose a nullspace trust region adaptation method embedded within the sub-problem solver towards a comprehensive hierarchical step-filter.
We demonstrate the computational efficiency of the hierarchical solver on typical test functions like the Rosenbrock and Himmelblau's functions, inverse kinematics problems and prioritized discrete non-linear optimal control.   
\end{abstract}


\begin{amscode}
90C29, 65F50, 49M15, 49M37, 90C55
\end{amscode}

\section{Introduction} 

\subsection{Context and contribution}
Lexicographic multi-objective optimization (LMOO) is the hierarchical stacking of objectives, where each objective is infinitely more important than the succeeding ones. It has its early beginnings in linear programming~\citep{Sherali1983}. 
LMOO has been extended to non-linear objective functions~\citep{Evtushenko2014,Lai2021,abu2023} and has been applied for example to road pricing~\citep{Zhong2022}. 
In recent years, lexicographic optimization for linear least-squares programming has received considerable attention, especially in the robot control community. 
The authors in~\citep{Kanoun2009} first enabled the handling of infeasible inequality constraints on any priority level. 
This is in contrast to LMOO which only considers inequality constraints on the state vector~\citep{Lai2021}. 
The work in~\citep{escande2014} proposed a very efficient solver for resolving these hierarchical least-squares programs (HLSP) based on a reduced Hessian formulation~\citep{nocedal2006}. 
In this work we aim to contextualize HLSP's within hierarchical non-linear programming~\eqref{eq:HNLP}, and specifically non-linear hierarchical least-squares programming~\ref{eq:NL-HLSP}. 
We propose a sequential hierarchical least-squares programming solver (S-HLSP) with trust region (TR) and hierarchical step-filter (HSF) with global convergence properties (Sec.~\ref{sec:shlsp}).
 The algorithm is based on the sparse HLSP sub-problem solver s-$\mathcal{N}$IPM-HLSP tailored to discrete optimal control (Sec.~\ref{sec:turnback} and Alg.~\ref{alg:solvehlsp}). Unlike LMOO, we are able to handle infeasible inequality constraints on any priority level.

%

\subsection{Non-linear programming}

Constrained non-linear programs (NLP) can be solved
by including a  penalty term on the constraints in the objective function. The resulting unconstrained problem can be resolved by the Newton's method~\citep{nocedal2006}. Different forms of the penalty term have emerged. The quadratic penalty method~\citep{courant1960} adds a quadratic term. Augmented Lagrangian methods~\citep{hestenes1969} include a more refined term by maintaining an estimate of the Lagrange multipliers associated with the constraints. Another penalty function can be considered the log-barrier function~\citep{gould1989} which penalizes constraints when they move towards the boundaries of the feasible domain. This penalty term is commonly used in the primal-dual interior-point methods (IPM)~\citep{ipopt}.

Another approach to solve NLP's is
sequential quadratic programming (SQP)~\citep{CHIANG1994,gill2005}.	Here, the non-linear optimality conditions are linearized. The linearized \textit{sub-problem} is iteratively solved until a KKT point is reached. 
Second order optimality conditions require the Hessian to be positive definite, for example by regularization~\citep{Higham1986}. A positive definite approximation can be obtained for example by the BFGS algorithm~\citep{Broyden1970}. SQP has been proposed in connection with a trust region constraint~\citep{Sun2019} which limits the magnitude (with respect to some norm) of the linearized step. This ensures that the approximated quadratic model represents the original problem sufficiently well within this region. 
Merit functions can be used in order to guide the algorithm to optima in case of non-convex optimization~\citep{han1975}. They balance sufficient reduction of the objective and feasibility of the constraints in order to approach an optimal point of the original non-linear problem. One disadvantage of this approach is the need for choosing penalty parameters. Filter methods circumvent this by maintaining a list of accepted filter points with respect to some constraint feasibility and sufficient objective reduction criterion~\citep{fletcher2002b}. A filter method with global convergence properties for the interior point method has been proposed in~\citep{ulbrich2004}.

\subsection{Sequential hierarchical least-squares programming}
S-HLSP (as a sub-form of SQP) for solving NL-HLSP's has been applied in numerous works, especially concerning real-time robot control (high frequency control with limited computation time for each control cycle)~\citep{pfeiffer2023}. The authors in~\cite{Kanoun2009} proposed to solve a cascade of least-squares programs (LSP) to solve the arising HLSP sub-problems. The resulting linear step of these~\ref{eq:hlsp} sub-problems is used to make an approximated step in the original non-linear problem. The authors in~\citep{escande2014} provided a very efficient solver for the~\ref{eq:hlsp} sub-problems based on the reduced Hessian formulation. It solves progressively smaller (and therefore cheaper) problems as it progresses through the hierarchy by variable elimination resulting from nullspace projections. This is one of the fundamental concepts in efficient hierarchical programming. While this solver is based on the active-set method, a solver based on the interior-point method has been proposed~\citep{pfeiffer2021}, which provides higher levels of algorithm stability for ill-conditioned problems. 

\subsection{Our contribution to sequential hierarchical least-squares programming}

Two issues can be identified with the current state-of-the-art of S-HLSP:
\begin{itemize}
	\item Specific application to real-time robot control without consideration of convergence requirements.
	\item The~\ref{eq:hlsp} sub-problem solvers are not adapted to sparse problem formulations.
\end{itemize}

The first point was addressed in~\cite{pfeiffer2018,pfeiffer2023} by contextualizing S-HLSP's within NLP, for example by proposing a real-time suitable trust-region adaptation method (that avoids expensive re-calculations of the~\ref{eq:hlsp} sub-problems) and by formulating the hierarchical Newton's method. However, the convergence properties were not investigated. In this work, we propose a S-HLSP with trust-region and HSF with global convergence properties (see Sec.~\ref{sec:shlsp}).

The second point becomes relevant when solving discrete optimal control problems. Here, the sub-problems exhibit a banded structure which, when properly exploited, only leads to a linear instead of cubic increase of the computational complexity with the length of the control horizon~\citep{wangboyd2010}. 
A block-wise Quasi-Newton method for banded Jacobian updates in discrete non-linear optimal control has been proposed for example in~\citep{Hespanhol2021}.
Sparsity preserving nullspace bases for reduced Hessian based solvers have been proposed in numerous works~\citep{topcu1979,Yang2019,pfeiffer2021b}.
Here, we propose the sparse HLSP solver s-$\mathcal{N}$IPM-HLSP based on the IPM and a computationally efficient implementation of the turnback algorithm (see Sec.~\ref{sec:turnback}).
The turnback algorithm has been first conceived in~\citep{topcu1979} and has been addressed in numerous work, for example with regards to computational efficiency~\citep{berry1985}, preserving higher levels of sparsity~\citep{Gilbert1987} or in the context of embedded control~\citep{dang2017}. 

Our developments are tested on typical test functions like the Rosenbrock and Himmelblau's functions, inverse kinematics problems and prioritized discrete optimal control, see Sec.~\ref{sec:eval}.   

\nomenclature[01]{Outer iteration}{(or just iteration, iter.) S-HLSP iteration $k$}
\nomenclature[02]{Inner iter.}{Iter. $\iota$ of HLSP sub-problem solver}
\nomenclature[03]{Control iter.}{Collocation point along the horizon of discrete optimal control problems}

\nomenclature[04]{\(l\)}{Current priority level}
\nomenclature[05]{\(p\)}{Overall number of priority levels, excluding the trust region constraint on $l=0$}	
\nomenclature[06]{\(n\)}{Number of variables}
\nomenclature[07]{\(m\)}{Number of constraints}

\nomenclature[09]{\(x\in\mathbb{R}^{n}\)}{Primal of S-HLSP}
\nomenclature[10]{\(\Delta x\in\mathbb{R}^{n}\)}{Primal step of HLSP sub-problem}
\nomenclature[11]{\(\Delta x^l\in\mathbb{R}^{n}\)}{Primal step of level $l$ of HLSP sub-problem}
\nomenclature[12]{\(\Delta x_l\in\mathbb{R}^{n}\)}{Sum of primal steps of levels $1$ to $l$ of HLSP sub-problem}
\nomenclature[13]{\(\Delta z^l\)}{Primal nullspace step of level $l$ of HLSP sub-problem}
\nomenclature[14]{\(\Delta \hat{x}_{\iota}\)}{Sub-step of inner iter. ${\iota}$ of HLSP sub-problem}

\nomenclature[15]{\(f(x)\in\mathbb{R}^{m}\)}{(Non-linear) task function}
\nomenclature[17]{\(J(x)\in\mathbb{R}^{m\times n}\)}{Task function Jacobian}
\nomenclature[18]{\(H(x)\in\mathbb{R}^{n\times n}\)}{Task function Hessian}	
\nomenclature[20]{\(\hat{H}_l\in\mathbb{R}^{n\times n}\)}{Hierarchical Hessian of level $l$}
\nomenclature[21]{\(\hat{R}_l\in\mathbb{R}^{n\times n}\)}{Factor of hierarchical Hessian of level $l$}

\nomenclature[22]{\(\mathbb{E}_l\)}{ Set of $m_{\mathbb{E}}$ equality constraints (eq.)   of level $l$}
\nomenclature[23]{\(\mathbb{I}_l\)}{ Set of $m_{\mathbb{I}}$ inequality constraints (eq.)  of level $l$}
\nomenclature[24]{\(\mathcal{I}_l\)}{ Set of $m_{\mathcal{I}}$ inactive inequality constraints (ineq.)  of level $l$ }
\nomenclature[25]{\(\mathcal{A}_l\)}{ Set of $m_{\mathcal{A}}$ active equality and inequality constraints of level $l$}
\nomenclature[26]{\(\mathbb{E}_{\cup l}\)}{ Set union $\mathbb{E}_{\cup l}\coloneqq\bigcup_{i=1}^l \mathbb{E}_i= \mathbb{E}_1 \cup \cdots \cup \mathbb{E}_l$ with $m_{{\mathbb{E}}_{\cup l}}$ constraints}

\nomenclature[30]{\(v_l\in\mathbb{R}^{m_l}\)}{Slack variable of (non-linear) task functions $f_l$ on level $l$ }
\nomenclature[31]{\(v^*_l\in\mathbb{R}^{m_l}\)}{Optimal slack variable of level $l$}

\nomenclature[32]{\(f_l^+\)}{Violation of active constraints on level $l$}
\nomenclature[33]{\(h_{\cup l}(x)\in\mathbb{R}^{m_{\cup l}}\)}{Measure of constraint infeasibility from levels $1$ to $l$}

\nomenclature[34]{\(\mathcal{L}\)}{Lagrangian}
\nomenclature[35]{\(\lambda_{\mathbb{E}}\in\mathbb{R}^{m_{\mathbb{E}}}\)}{Lagrange multipliers with respect to the constraint set $\mathbb{E}$}
\nomenclature[36]{\(\lambda_{l}\in\mathbb{R}^{m_{\cup l-1}}\)}{Lagrange multipliers of level $l$ with respect to levels 1 to $l-1$}
\nomenclature[37]{\(\lambda_{\cup p}\in\mathbb{R}^{m_{\cup p}\times p}\)}{Lagrange multipliers of all levels}

\nomenclature[40]{\(A_{\mathbb{E}}\in\mathbb{R}^{m_{\mathbb{E}}\times n}\)}{ Matrix  representing a set $\mathbb{E}$  of $m_{\mathbb{E}}$ linear constraints }
\nomenclature[41]{\(b_{\mathbb{E}}\in\mathbb{R}^{m_{\mathbb{E}}}\)}{ Vector representing a set $\mathbb{E}$ of $m_{\mathbb{E}}$ linear constraints }
\nomenclature[42]{\(A_{i:j}\)}{Matrix col. range from col. $i$ to and incl. $j$}
\nomenclature[43]{\(b_{i}\) or \(b_{i:j}\)}{Vector entry $i$ or range $i$ to and incl. $j$}
\nomenclature[44]{\(Z\)}{Nullspace basis of matrix $A$ with $AZ=0$}
\nomenclature[45]{\(N_l\)}{Accumulated nullspace basis $N_l = Z_1 \dots Z_l$}

\nomenclature[46]{\(\hat{\parm}\)}{HLSP value equivalent $\hat{\parm}$ of NL-HLSP value $\parm$}

\nomenclature[50]{\(\hat{w}\in\mathbb{R}^{m}\)}{Interior point slack variable of level}
\nomenclature[51]{\(\mu\)}{Interior point duality measure}
\nomenclature[52]{\(\sigma\)}{Interior point centering parameter}

\nomenclature[60]{\(\rho\)}{Trust region radius}
\nomenclature[61]{\(\rho_{\max,l}\)}{Optimal trust region radius of level $l$}
\nomenclature[62]{\(\epsilon\)}{Hessian augmentation threshold}
\nomenclature[63]{\(\chi\)}{Step threshold on $\Delta x$ or $\Delta x_l$}
\nomenclature[64]{\(\xi\)}{Constraint activation threshold}
\nomenclature[65]{\(nnz(M)\)}{Number of non-zeros of matrix $M$}
\nomenclature[66]{\(d(M)\)}{Density $nnz/(mn)$ of matrix $M\in\mathbb{R}^{m\times n}$ }

\printnomenclature[3cm]

\section{Problem definition and background}

\subsection{Hierarchical non-linear programming as hierarchical non-linear least-squares programming}

In this work, we consider hierarchical non-linear programs (H-NLP) with $p$ levels of the form ($\lexmin$: lexicographically minimize)
	\begin{align}
	\lexmin_x\qquad &f_{\mathbb{E}_1}(x), \dots , f_{\mathbb{E}_p}(x)
	\label{eq:HNLP}\tag{H-NLP}\\
	\text{s.t}\qquad& f_{\mathbb{I}_{\cup p}}(x) \leq 0\nonumber
\end{align}
$x\in\mathbb{R}^n$ is a variable vector. 
On each level $l$, the \textit{objective} (or equality constraint $\mathbb{E}$) $f_{\mathbb{E}_l}\in\mathbb{R}^{m_{\mathcal{E}_l}}$ needs to be minimized while maintaining the optimality of the infinitely more important \textit{constraints} $f_{\mathbb{E}_{\cup l-1}}\in\mathbb{R}^{m_{\mathbb{E}_{\cup l-1}}}$ of levels $1$ to $l-1$. $\mathbb{E}_{\cup l-1}$ indicates the set union $\mathbb{E}_{\cup l-1} \coloneqq \bigcup_{i=1}^{l-1}\mathbb{E}_i = \mathbb{E}_1 \cup \cdots \cup \mathbb{E}_{l-1}$ of the sets of equality constraints of levels 1 to $l-1$. In the meantime, inequality constraints $f_{\mathbb{I}_{\cup p}}\in\mathbb{R}^{m_{\mathbb{I}_{\cup p}}}$  need to be respected. 
This is an extension to LMOO since in classical LMOO inequality constraints are limited to the state vector~\citep{Lai2021}.
Note that in hierarchical programming the notation of \textit{objective} (on a level $l$) and \textit{constraint} (on levels $1$ to $l-1$) is softened since every priority level except the first and last one can be considered both a constraint and an objective level at the same time. In the following, we omit the dependency of $x$ for better readability.

The above programming only holds when an $x$ exists such that the inequality constraints $f_{\mathbb{I}_{\cup p}}(x)\leq 0$ are fulfilled / feasible / compatible~\citep{fletcher2002b}. 
In order to relax the problem, we solve the~\ref{eq:HNLP} with respect to the $\ell_2$-norm by introducing slack variables $v$~\citep{pfeiffer2023}. This effectively casts the hierarchical non-linear programming into a hierarchical non-linear least-squares programming (NL-HLSP). The resulting NL-HLSP is given as follows:
\begin{align}
	\lexmin_{x,v_{\mathbb{E}_{1}},\dots,v_{\mathbb{I}_{p}}}\qquad &  \frac{1}{2}\Vert v_{\mathbb{E}_{1}}\Vert^2_2 +  \frac{1}{2}\Vert v_{\mathbb{I}_{1}}\Vert^2_2, \dots ,  \frac{1}{2}\Vert v_{\mathbb{E}_{p}}\Vert^2_2 +  \frac{1}{2}\Vert v_{\mathbb{I}_{p}}\Vert^2_2 \label{eq:nlhlsplexmin}\\
	\text{s.t}\qquad& f_{\mathbb{E}_{\cup p}} = v_{\mathbb{E}_{\cup p}}\nonumber\\
		\qquad& f_{\mathbb{I}_{\cup p}} \leq v_{\mathbb{I}_{\cup p}}\nonumber
\end{align}
This formulation with slacks has the advantage that infeasible inequality constraints can be handled on any priority level~\citep{Kanoun2009}. This also discards the need for a restoration phase~\citep{Choong2007} or feasibility refinement~\citep{Geffken2017} as it is common in SQP algorithms. 

\begin{remark}
	A drawback of the~\ref{eq:NL-HLSP} formulation is that its minima are not necessarily minima of the original~\ref{eq:HNLP}. An example would be minimization in the $\ell_1$-norm for sparse programming (basis pursuit, Lasso)~\citep{Candes2007}, which is not equivalently solvable by least-squares programming. However, in our primary use-case of robot planning and control the least-squares formulation is usually sufficient as has been demonstrated in numerous engineering works (\cite{pfeiffer2023} and references therein).
\end{remark}
\normalsize

Throughout the remainder of this article, we consider the following optimization problem which is sequentially solved for each priority level $l=1,\dots,p$, and which is equivalent to~\eqref{eq:nlhlsplexmin}:
\begin{align}
	\mini_{{x},v_{\mathbb{E}_{l}},v_{\mathbb{I}_{l}}} \quad & \frac{1}{2} \left\|v_{\mathbb{E}_{l}}\right\|^2_2 + \frac{1}{2} \left\|v_{\mathbb{I}_{l}}\right\|^2_2 \qquad\quad\hspace{8pt}l = 1,\dots,p
	\label{eq:NL-HLSP}\tag{NL-HLSP}\\
	\mbox{s.t.} \quad & f_{\mathbb{E}_{l}} = v_{\mathbb{E}_{l}}\nonumber\\
	\qquad& f_{\mathbb{I}_{l}} \leq v_{\mathbb{I}_l}\nonumber\\
	\qquad& f_{\mA_{\cup l-1}} = v_{\mA_{\cup l-1}}^*\nonumber\\
\qquad& f_{\mI_{\cup l-1}} \leq 0\nonumber
\end{align}
The slack variables $v_{\mA_{\cup l-1}}^*$ are the optimal ones identified for the higher priority levels $1$ to $l-1$. The \textit{active set} $\mA_{\cup l-1}$ contains all constraints that were active at convergence of levels $1$ to $l-1$. This includes all equality constraints $\bE_{\cup_{l-1}}$, and furthermore all violated / infeasible or saturated inequality constraints of $\bI_l$. In a similar vein, the \textit{inactive set} $\mI_{\cup l-1}$ contains all the remaining feasible inequality constraints of the sets $\bI_{\cup l-1}$.

This way of resolving a~\ref{eq:HNLP} / LMOO is commonly referred to as
\textit{preemptive} scheme. This is in contrast to the nonpreemptive scheme where each objective is weighted against each other with finite weights~\citep{Cococcioni2018}.

\subsection{Sequential hierarchical least-squares programming}

The Lagrangian $\mathcal{L}_l$ of level $l$ of the~\ref{eq:NL-HLSP} writes as
\begin{align}
	\mathcal{L}_l(x,v_l,\lambda_l) = \frac{1}{2}v_{\mathbb{E}_{l}}^Tv_{\mathbb{E}_{l}} &+ \frac{1}{2}v_{\mathbb{I}_{l}}^Tv_{\mathbb{I}_{l}} + {\lambda}_{\mathbb{E}_{l}}^T(f_{\mathbb{E}_{l}} - v_{\mathbb{E}_{l}}) + {\lambda}_{\mathbb{I}_{l}}^T(f_{\mathbb{I}_{l}} - v_{\mathbb{I}_{l}}) 	\label{eq:lagrangianGenHier}\\
	&+ {\lambda}_{\mA_{\cup l-1}}^T({f}_{\mA_{\cup l-1}} - {v}_{\mA_{\cup l-1}}^* ) + {\lambda}_{\mI_{\cup l-1}}^T {f}_{\mI_{\cup l-1}} \nonumber
\end{align}
${\lambda}_{\mathbb{E}_{l}}$, ${\lambda}_{\mathbb{I}_{l}}$, ${\lambda}_{\mA_{\cup l-1}}$ and ${\lambda}_{\mI_{\cup l-1}}$ are the Lagrange multipliers associated with the constraint sets $\mathbb{E}_{l}$, $\mathbb{I}_{l}$, $\mA_{\cup l-1}$ and $\mI_{\cup l-1}$, respectively. In the following, we use $\lambda_l$ interchangeably for the above and only specify the constraint set if necessary. Furthermore, the Lagrange multipliers $\lambda_{\cup p}\in\mathbb{R}^{m_{\cup p},p}$ contain the Lagrange multipliers $\lambda_l$ of all levels $l=1,\dots,p$. Similarly, $v_l$ represents $v_{\mathbb{E}_{l}}$ and $v_{\mathbb{I}_{l}}$.
We apply the hierarchical Newton's~\citep{pfeiffer2023} or Quasi-Newton method~\citep{pfeiffer2018} or Gauss-Newton algorithm to the (non-linear) first-order optimality conditions $\nabla \mathcal{L}_l(x,v_l,\lambda_l)=0$ of the~\ref{eq:NL-HLSP}. 
 The linearized problem writes as
\begin{align}
	\mini_{\Delta x,\hat{v}_{\mathbb{E}_l},\hat{v}_{\bI_l}}& \qquad \frac{1}{2}\Vert \hat{v}_{\mathbb{E}_l} \Vert^2_2 + \frac{1}{2}\Vert \hat{v}_{\bI_l} \Vert^2_2 \label{eq:hlsp}\tag{HLSP}\\
	\text{s.t.}
	& \qquad A_{\mathbb{E}_l}\Delta x - b_{\mathbb{E}_l} = \hat{v}_{\mathbb{E}_l}  \qquad l=0,\dots,p \nonumber\\
	& \qquad A_{\bI_l}\Delta x - b_{\bI_l} \leq \hat{v}_{\bI_l}\nonumber\\
	& \qquad A_{\mA_{\cup l-1}}\Delta x - b_{\mA_{\cup l-1}} = \hat{v}_{\mA_{\cup l-1}}^*\nonumber\\
	& \qquad A_{\mI_{\cup l-1}}\Delta x - b_{\mI_{\cup_{l-1}}} \leq 0\nonumber
\end{align}
$\hat{v}_l$ indicates that the slacks (and similarly the Lagrange multiplier estimate $\hat{\lambda}$) of the linearized problem do not necessarily coincide with the ones from the original one~\eqref{eq:NL-HLSP} except at S-HLSP convergence~\citep{fletcher2002b}.
The constraint matrices represent the linearizations of $f$ at a given $x$ and are composed as follows:
\begin{align}
	A_{\mathbb{E}_l} \coloneqq \BIN J_{\mathbb{E}_l} \\ R_l \BOUT \qquad &\text{and} \qquad b_{\mathbb{E}_l} \coloneqq \BIN f_{\mathbb{E}_l} \\ 0 \BOUT\\
	A_{\mathbb{I}_l} \coloneqq J_{\mathbb{I}_l}  \qquad &\text{and} \qquad b_{\mathbb{I}_l} \coloneqq f_{\mathbb{E}_l} \\
	A_{\mA_{\cup_{l-1}}} \coloneqq \BIN J_{\mathcal{A}_{l-1}} \\ {R}_{\cup_{l-1}} \BOUT \qquad &\text{and} \qquad b_{\mA_{\cup l-1}} \coloneqq\BIN f_{\mA_{\cup l-1}}  \\ 0 \BOUT\\
	A_{\mI_{\cup_{l-1}}} \coloneqq  J_{\mI_{l-1}}  \qquad &\text{and} \qquad b_{\mI_{\cup l-1}} \coloneqq f_{\mI_{\cup l-1}} 
\end{align}
$J$ are the Jacobians $J = \nabla_x f(x)$.
$R$ is a factor (for example Cholesky) of the positive-definite hierarchical Hessian $\nabla_{xx}^2\mathcal{L}=\hat{H} = R^T R$~\citep{pfeiffer2023}. It relies on the Hessian components  $H = \nabla^2_{xx} f(x)$ corresponding to each dimension of the functions $f$. 
Note that the Hessian components of the inequality constraints in $\mI_{\cup l-1}$ and ${\bI_l}$ are included in $R_{l-1}$ and $R_l$, respectively. $\hat{H}_{l}$ can be indefinite on construction but can be regularized to a positive-definite matrix for example by Schur regularization or the Higham algorithm~\citep{Higham1986}.
The matrix can also be approximated by a positive-definite BFGS update scheme as proposed in~\cite{pfeiffer2018}. Similarly to the latter, we set the Hessian to zero whenever the linearized model is accurate enough, that is $\Vert \hat{v} \Vert^2_2 \leq \epsilon$ ($\epsilon$ is a small numerical threshold; note that this condition is usually not fulfilled if the current state $x$ is far away from the optimum $x^*$, or the Jacobians $J$ are rank deficient and only a bad Hessian approximation exists). This avoids the unnecessary occupation of variables as the Hessian is full-rank on the variables that are in the corresponding constraint $f$.

The~\ref{eq:hlsp} can be solved by any of the solvers~\cite{Kanoun2009,escande2014,pfeiffer2021} for the primal $\Delta x$ and the dual / Lagrange multiplier estimate $\hat{\lambda}$ (see Sec.~\ref{sec:hlsp} for more details).
We do this repeatedly in a sequential hierarchical least-squares programming (S-HLSP)~\citep{pfeiffer2023} until some convergence criteria of the original non-linear~\ref{eq:NL-HLSP} is met. $k$ is the current S-HLSP (outer) iteration with $x_{k+1} = x_k + \Delta x_k$.
The HLSP at the current state $x_k$ is hereby referred to as HLSP \textit{sub-problem}. 
 We restrict the primal step by $\Vert\Delta x\Vert_{\infty} < \rho$. 
This \textit{trust region} constraint is a linear bound constraint ($\Vert \Delta x \Vert_\infty \leq \rho$, $\rho$ is the trust region radius) on the first level $l=0$ (note that we start the original non-linear problem~\ref{eq:NL-HLSP} from level $l=1$). It limits the discrepancy between the original~\ref{eq:NL-HLSP} and its~\ref{eq:hlsp} approximation. The trust region radius can be adapted for example as presented in~\citep{fletcher2002b} in the context of a step-filter for SQP (SQP-SF). In this work, we propose a hierarchical step-filter (HSF) for trust region adaptation based on the SQP step-filter, see Sec.~\ref{sec:shlsp}.
For an overall overview of the S-HLSP see Fig.~\ref{fig:hfilter}.

For later reference, we define 
the step $\Delta x_l$ of level $l$ as
\begin{equation}
	\Delta x_l = \sum_{j=0}^{l} \Delta x^j \qquad \text{with} \qquad \Delta x = \sum_{j=0}^{p} \Delta x^j \label{eq:stepdef}
\end{equation}
This means that $\Delta x_l$ represents the primal step resulting from the resolution of the levels $j=1,\dots,l$ with primal sub-steps $\Delta x^j$ while neglecting the contributions of the levels $l+1$ to $p$ (also see Alg.~\ref{alg:solvehlsp}). Furthermore, we define
the predicted objective reduction
\begin{align}
	\Delta q_l &= 2b_{\mathbb{E}_l}^T A_{\mathbb{E}_l}  \Delta x_k - \Delta x_k^T A_{\mathbb{E}_l}^T A_{\mathbb{E}_l} \Delta x_k + 2b_{\bI_l}^T A_{\bI_l} \Delta x_k - \Delta x_k^T A_{\bI_l}^T A_{\bI_l} \Delta x_k	\nonumber
\end{align}
and the actual reduction
\begin{equation}
	\Delta f_l = \left\Vert f_l^+(x_k) \right\Vert^2_2 - \left\Vert f_l^+(x_k + \Delta x_k) \right\Vert^2_2
\end{equation}
where
\begin{equation}
f_{l}^+ \coloneqq \BIN f_{\bE_l} \\ \max(0,f_{\bI_l}) \BOUT
\end{equation} 
 Note that we use the squared norm for $\Delta f_l$ in order to be consistent with $\Delta q_l$.

\subsection{Reduced Hessian based interior-point method for hierarchical least-squares programming}

\label{sec:hlsp}
\begin{algorithm}[t!]
	\caption{s-$\mathcal{N}$IPM-HLSP with NSTRA}
	\label{alg:solvehlsp}
	\begin{algorithmic}[1]
		\Statex\textbf{Input:} {$x$, $\rho_{1:p}$}
		\Statex \textbf{Output:} $\Delta x$, $\hat{v}^*_{\mA_{\cup p}}$, $\hat{\lambda}_{\mA_{\cup p}}$
		\State ${\iota}=1$ 
		\State $\Delta x_{\iota} = 0$
		\For{$l=1:p$}
		\State Set $\rho = \rho_l$
		\State $\Delta x^l = 0$
		\While{$\Vert \nabla\hat{\mathcal{L}}_l \Vert_2 \neq 0$ of~\ref{eq:hlsp}}
		\State $\Delta \hat{z}^l_{\iota}\leftarrow$ Solve~\eqref{eq:HLSPNeNmethod} at $x$ with $\rho$
		\State $\hat{\Delta} x^l_{\iota} = N_{l-1}\hat{\Delta} z^l_{\iota}$
		\State Compute dual step $\Delta v_l$, $\Delta w_l$ and $\Delta\hat{\lambda}_{l,\iota}$ from substitutions of $\nabla\hat{\mathcal{L}}_l = 0$ of~\ref{eq:hlsp}
		\State $\alpha\leftarrow$ line search for dual feasibility~\eqref{eq:linesearch_li},~\eqref{eq:linesearch_l-1i}
		\State Primal step: ${\Delta} x^l = {\Delta} x^l + \alpha \hat{\Delta} x^l_{\iota}$
		\State Dual step: $\parm_{{\iota}+1} = \parm_{{\iota}} + \alpha\Delta \parm_{{\iota}}$ for new $\hat{v}_l$, $\hat{w}_l$, $\hat{\lambda}_l$
		\State $\iota = \iota + 1$
		\EndWhile
		\State $\mA_{l^*}$, $\mI_{\cup l^*}\leftarrow {\mI_{\cup l-1}}$ depending on condition~\eqref{eq:addcondl-1}
		\State $\hat{v}^*_{\mA_l^*} \leftarrow v_{\mI_{l-1}}$ depending on condition~\eqref{eq:addcondl-1}
		\State $Z_{l^*}\leftarrow$ Alg.~\ref{alg:recturnback} with $A_{\umA_{l^*}}N_{l-1}$ as input
		\State $N_{l^*} = N_{l-1}Z_{l^*}$
		\State $\mA_{l}\leftarrow$ all equality constraints ${\mathbb{E}_l}$
		\State $\hat{v}^*_{\mA_l} \leftarrow v_{\bE_l}$
		\State $\mA_{l}$, $\mI_{\cup l}\leftarrow$ ${\bI_l}$, ${\mI_{\cup l-1}}$ depending on condition~\eqref{eq:addcondl}
		\State $\hat{v}^*_{\mA_l} \leftarrow \BIN \hat{v}_{\mA_l}^{*T} &  v_{\bI_l}^T & v_{\mI_{l-1}}^T \BOUT^T$ depending on condition~\eqref{eq:addcondl}
		\State $Z_l\leftarrow$ Alg.~\ref{alg:recturnback} with $A_{\umA_{l}}N_{l^*}$ as input
		\State $N_l = N_{l^*}Z_l$
		\EndFor
		\State \Return $\Delta x$, $\hat{v}^*_{\mA_{\cup p}}$, $\hat{\lambda}_{\mA_{\cup p}}$
	\end{algorithmic}
\end{algorithm}

We solve the~\ref{eq:hlsp} by means of a sparse version of the reduced Hessian based interior point method presented in~\cite{pfeiffer2021}. An algorithmic overview with some extensions developed throughout this article (nullspace trust region adaptation (NSTRA), see Sec.~\ref{sec:nstradapt}) is given in Alg.~\ref{alg:solvehlsp}. For a detailed description of the algorithm the reader is referred to the original work in~\cite{pfeiffer2021}.

	First, two positive slack variables, $\hat{w}_{\mathbb{I}_l}$ for the inequality constraints on the current level $l$ and $\hat{w}_{\mI_{\cup l-1}}$ for the inactive inequality constraints of the previous levels, are introduced to the~\ref{eq:hlsp}. They are penalized by the log function to maintain positive values (`log-barrier'):
\begin{align}
	\mini_{\Delta x,\hat{v}_{\mathbb{E}_l},\hat{v}_{\mathbb{I}_l},\hat{w}_{\mathbb{I}_l},\hat{w}_{\mI_{\cup l-1}}} & \qquad \frac{1}{2}\Vert \hat{v}_{\mathbb{E}_l} \Vert^2_2 + \frac{1}{2}\Vert \hat{v}_{\mathbb{I}_l} \Vert^2_2 - \sigma_{\mathbb{I}_l}\mu_{\mathbb{I}_l}\sum\log(\hat{w}_{\mathbb{I}_l}) 
	- \sigma_{\mI_{\cup l-1}}\mu_{\mI_{\cup l-1}}\sum\log(\hat{w}_{\mI_{\cup l-1}})\nonumber\\
	\text{s.t.}
	& \qquad A_{\mathbb{E}_l}\Delta x - b_{\mathbb{E}_l} = \hat{v}_{\mathbb{E}_l}\nonumber\\
	& \qquad A_{\mathbb{I}_l}\Delta x - b_{\mathbb{I}_l} - \hat{v}_{\mathbb{I}_l} = \hat{w}_{\mathbb{I}_l}\nonumber\\
	&\qquad \hat{w}_{\mathbb{I}_l} \geq 0 \nonumber\\
	& \qquad A_{\mA_{\cup l-1}}\Delta x - b_{\mA_{\cup l-1}} = \hat{v}_{\mA_{\cup l-1}}^*\nonumber\\
	& \qquad A_{\mI_{\cup l-1}}\Delta x - b_{\mI_{\cup l-1}} = \hat{w}_{\mI_{\cup l-1}}\nonumber\\
	&\qquad \hat{w}_{\mI_{\cup l-1}} \geq 0\label{eq:ipmhqpopt} 
\end{align}
$\mu$ and  $\sigma$ are the duality measure and the centering parameter, respectively~\citep{vanderbei2013}.
The Lagrangian $\hat{\mathcal{L}}_l$ of the barrier~\ref{eq:hlsp} writes as 
	\begin{align}
	\hat{\mathcal{L}}_l &\coloneqq \frac{1}{2}\Vert \hat{v}_{\mathbb{E}_l} \Vert^2_2 + \frac{1}{2}\Vert \hat{v}_{\mathbb{I}_l} \Vert^2_2 - \sigma_{\mathbb{I}_l}\mu_{\mathbb{I}_l}\sum\log(\hat{w}_{\mathbb{I}_l}) -\sigma_{\mI_{\cup l-1}}\mu_{\mI_{\cup l-1}}\sum\log(\hat{w}_{\mI_{\cup l-1}})\nonumber\\
	& + \lambda_{\mathbb{E}_l}^T(A_{\mathbb{E}_l}\Delta x - b_{\mathbb{E}_l} - \hat{v}_{\mathbb{E}_l}) + \lambda_{\mathbb{I}_l}^T(A_{\mathbb{I}_l}\Delta x - b_{\mathbb{I}_l} - \hat{v}_{\mathbb{I}_l} - \hat{w}_{\mathbb{I}_l})\nonumber\\
	&  + \lambda_{\mA_{\cup l-1}}^T (A_{\mA_{\cup l-1}}\Delta x - b_{\mA_{\cup l-1}} - \hat{v}_{\mA_{\cup l-1}}^*)+ \lambda_{\mI_{\cup l-1}}^T(A_{\mI_{\cup l-1}}\Delta x - b_{\mI_{\cup l-1}} - \hat{w}_{\mI_{\cup l-1}})
\end{align}
The Newton's method is then applied to the first-order optimality condition $\nabla\hat{\mathcal{L}}_l=0$ of the barrier~\ref{eq:hlsp}. After appropriate substitutions of the dual steps $\Delta \hat{v}_l$, $\Delta \hat{w}_l$ (which represents both $\Delta \hat{w}_{\mathbb{I}_l}$ and $\Delta \hat{w}_{\mI_{\cup l-1}}$) and $\Delta \hat{\lambda}_l$, the HLSP solver repeatedly solves the Newton step for the primal step $z^l_{\iota}$ (for example by Cholesky decomposition; inner iteration ${\iota}$)
\begin{align}
	N_{l-1}^TC_{l,\iota} N_{l-1} 
	\hat{\Delta} z^l_{\iota}
	=
	N_{l-1}^Tr_{l,\iota}
	\label{eq:HLSPNeNmethod}
\end{align}
$l$ is the current level of the~\ref{eq:hlsp} that is resolved. $\hat{\Delta} z^l_i$ is a nullspace step with the projection $\hat{\Delta} x^l_{\iota} = N_{l-1}\hat{\Delta} z^l_{\iota}$. $N_{l-1}$ is a nullspace basis of the current active-set $\mA_{\cup l-1}$ of lower priority levels $1$ to $l-1$. 
$C_l$ is the weighted normal form $C_l=A^T\hat{\Lambda}A$ of a set of constraints $A$ representing the constraint matrices $A=\BIN A_{\mathbb{E}_l}^T &  A_{\bI_l}^T & A_{\mI_{\cup l-1}}^T\BOUT^T$. $\hat{\Lambda}$ is a matrix representing the current dual $\hat{\lambda}$ of the problem. $N_{l-1}^TC_lN_{l-1}$ is the corresponding reduced Hessian~\citep{jaeger1997} which, with the right choice of nullspace basis $N$, is either lower in dimension (dense programming) or in its number of non-zeros (sparse programming); $r$ is the right hand side. The variables $N$, $C$ and $r$ are variant depending on the current state $x_k$, the current step $\Delta x_{\iota}$ and / or the current Lagrange multiplier estimate $\hat{\lambda}_{l,\iota}$.

The IPM maintains a current update of the primal $\Delta x$ with $\Delta x_{{\iota}+1} = \Delta x_{\iota} + \alpha \Delta \hat{x}^l_{\iota}$. The dual variables $\hat{v}_l$, $\hat{w}_l$ and $\hat{\lambda}_{l,\iota}$ are updated with $\parm_{l,{\iota}+1} =  \parm_{l,\iota} + \alpha \Delta \parm_{l,\iota}$ where $\parm$ is a placeholder for the aforementioned variables.
$\alpha$ is a line search factor that maintains dual feasibility
	\begin{equation}
	\hat{v}_{\mathbb{I}_l} \leq 0 \qquad \text{and} \qquad
	\hat{w}_{\mathbb{I}_l} \geq 0\label{eq:linesearch_li}
\end{equation}
and
\begin{equation}
	\hat{\lambda}_{\mI_{\cup l-1}} \geq 0 \qquad \text{and} \qquad
	\hat{w}_{\mI_{\cup l-1}} \geq 0\label{eq:linesearch_l-1i}
\end{equation}
After level $l$ has converged with $\Vert \nabla\hat{\mathcal{L}}_l\Vert_2=0$, first the active constraints from $\mI_{\cup l-1}$ are added to the active set $\mA_{l^*}$ of the \textit{virtual} priority level $l^*$ via the conditions
	\begin{align}
	\hat{w}_{\mI_{\cup l-1}}(c) < \xi \qquad \text{ and } \qquad \hat{\lambda}_{\mI_{\cup l-1}}(c) > \xi  \label{eq:addcondl-1}
\end{align}
The nullspace basis $Z_{l^*}$ corresponding to $\mathcal{A}_{l^*}$ is computed (Alg.~\ref{alg:recturnback}) and used to augment $N_{l^*} = N_{l-1} Z_{l^*}$.
Secondly,
the active constraint set $\mA_{l}$ is determined according to the conditions
	\begin{equation}
	\hat{w}_{\mathbb{I}_l}(c) < \xi \qquad \text{ and } \qquad \hat{v}_{\mathbb{I}_l}(c) < -\xi \label{eq:addcondl}
\end{equation}
The nullspace basis $Z_{l}$ corresponding to $\mathcal{A}_{l}$ is computed (Alg.~\ref{alg:recturnback}) and used to augment $N_{l} = N_{l^*}  Z_{l}$.
Inactive constraints from $\mI_{\cup l-1}$ and $\mI_{l}$ are added to the inactive constraint set $\mI_{\cup l}$.
The notion of virtual priority levels ensures that the priority order between activated constraints from $\mI_{\cup l-1}$ and $\mI_l$ is observed.

 This procedure is repeated for all levels $l=1,\dots,p$ of the~\ref{eq:hlsp}.

\subsection{Sparsity in hierarchical least-squares programming for discrete optimal control}

The Jacobians $J$ (and similarly the Hessians) of the~\ref{eq:hlsp} can expose a banded structure as the result of discrete optimal control problems of the form
\begin{align}
	&\mini_{x,v_{l}} \quad  \frac{1}{2} \left\| {v}_{l} \right\|^2_2 \qquad l = 1,\dots,p  \label{eq:oc}\\
	&\mbox{s.t.} \quad  f_l({x_{\tau+1},x_\tau})  \leqq  {v}_{l,\tau} 	\qquad \tau = 1,\dots,T\nonumber\\
	& \qquad{f}_{\cup l-1}({x_{\tau+1},x_\tau})  \leqq  {v}_{\cup l-1,\tau}^* \nonumber
\end{align}
Here, $f$ represent both equality and inequality constraints which is indicated by the symbol $\leqq$.
The variable vector $x \coloneqq \BIN x_1 & \cdots & x_T\BOUT^T$
is composed of variables $x_{\tau}$ associated to each time-step $\tau=1,\dots,T$. $T$ is the length of the optimal control horizon. Similarly, $v_l$ takes the form $v_l\coloneqq\BIN v_{l,1}^T & \cdots & v_{l,T}^T \BOUT^T$.
Due to the dependency of subsequent variable segments, the Jacobians $J$ of the HLSP exhibit a banded structure which looks as follows
\begin{align}
	&J_l = 
	\BIN \nabla_{x_1} f_l(x_1, x_2) & \nabla_{x_2} f_l(x_1, x_2) & \cdots & 0 \\
	0 & \nabla_{x_2} f_l(x_2, x_3) & \cdots & 0 \\
	\vdots & \vdots & \ddots & \vdots \\
	0 & 0 & \cdots & \nabla_{x_T} f_l(x_{T-1}, x_T)
	\BOUT\label{eq:blockA}
\end{align}

In the HLSP solver~\citep{escande2014}, a dense nullspace basis $N$ was chosen. It is based on a rank-revealing QR decomposition with no but incidental sparsity maintenance. This can instead be achieved by nullspace bases based on the rank-revealing LU decomposition. Implementations exist that maintain sparse LU factors for example by adhering to the Markovitz criterion~\citep{markowitz1957}. While the resulting basis is sparse, higher level sparsity patterns like bands are not specifically considered and will be destroyed in the process. One basis that is able to maintain banded matrix sparsity is based on the turnback algorithm~\citep{berry1985} While this is not guaranteed, it has been shown to reliably produce banded nullspace bases~\citep{Yang2019}.
In a reduced Hessian based HLSP solver it is important to leverage this sparsity in the nullspace basis in order to remain computationally tractable, especially with long horizons (linear computational complexity in horizon length $T$, instead of cubic dependency for dense bases). In this article we present an efficient computation of turnback based nullspace bases (see Sec.~\ref{sec:turnback}). 

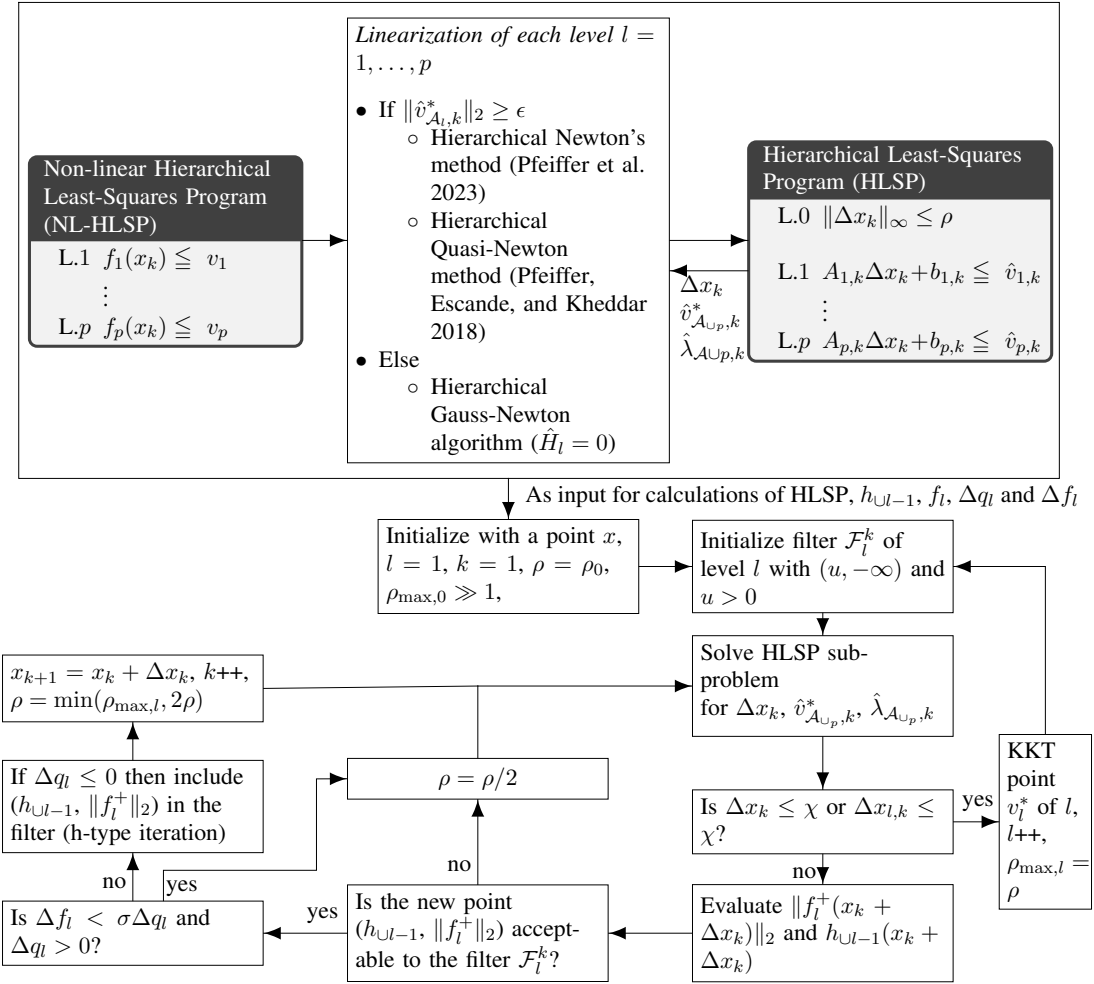
\begin{figure*}[htp!]
	\centering
	        \resizebox{\columnwidth}{!}{%
	\begin{tikzpicture}[line cap=rect]
		
		\node[align=left,text width=4.5cm] (NLHLSP) at (-0.75,-1) {
			\begin{tcolorbox}[title=Non-linear Hierarchical \\Least-Squares Program\\ \eqref{eq:NL-HLSP},boxsep=1pt,left=5pt,right=3pt,top=2pt,bottom=1pt] 
				\begin{enumerate}
					\item[L.$1$] $f_1(x_k) \leqq~v_1$\\
					{\centering$\vdots$}
					\item[L.$p$] $f_p(x_k) \leqq~v_p$
				\end{enumerate}	
			\end{tcolorbox}
		};
		
		\node[draw,align=left,text width=5cm] (Lin) at (4.85,-1) {
			\textit{Linearization of each level $l=1,\dots,p$}\\
			\begin{itemize}[leftmargin=*]
				\item If $\Vert \hat{v}_{\mA_{l},k}^*\Vert_2 \geq \epsilon$
				\begin{itemize}
					\item {{Hierarchical Newton's method~\citep{pfeiffer2023}}}
					\item Hierarchical Quasi-Newton method~\citep{pfeiffer2018}
				\end{itemize}
				\item Else 
				\begin{itemize}
					\item {Hierarchical Gauss-Newton algorithm ($\hat{H}_l = 0$)}
				\end{itemize}
			\end{itemize}
		};

		\node[text width=5cm] (HLSP) at (11.25,-1) {
			\begin{tcolorbox}[title=Hierarchical Least-Squares\\ Program \eqref{eq:hlsp},boxsep=1pt,left=5pt,right=3pt,top=2pt,bottom=1pt] 
				\begin{enumerate}
					\item[L.$0$] $\Vert\Delta x_k\Vert_{\infty} \leq \rho$\\
					\item[L.$1$] $A_{1,k}\Delta x_k+b_{1,k}\leqq~\hat{v}_{1,k}$\\
					{\centering$\vdots$}
					\item[L.$p$] $A_{p,k}\Delta x_k + b_{p,k}  \leqq~\hat{v}_{p,k}$
				\end{enumerate}	
			\end{tcolorbox}
		};
		
		\node[align=left] (feedback) at ($(HLSP.west) + (-0.425,-1.3)$) {$\Delta x_k$\\$\hat{v}_{\mA_{\cup p},k}^*$\\$\hat{\lambda}_{\mA{\cup p},k}$};
		
		\node[] (input) at ($(Lin.south) + (4.8,-0.55)$) { 
			As input for calculations of~\ref{eq:hlsp}, $h_{\cup l-1}$, $f_l$, $\Delta q_l$ and $\Delta f_l$
		};
		
		\node[draw,text width=4cm] (Init) at ($(Lin.south) - (0,1.7)$) {
			{Initialize with a point $x$, $l=1$, $k=1$, $\rho = \rho_0$, $\rho_{\max,0} \gg 1$, }
		};
		
		\node[draw,text width=4cm] (InitFilt) at ($(Init.east) + (3,0)$) {
			{Initialize filter $\mathcal{F}_l^k$ of level $l$ with $(u,-\infty)$} and $u>0$
		};
		
		\node[draw,text width=4cm] (sol) at ($(InitFilt.south) + (0,-1.2)$) {
			{Solve~\ref{eq:hlsp} sub-problem}\\ 
			for $\Delta x_k$, $\hat{v}_{\mA_{\cup_p},k}^*$, $\hat{\lambda}_{\mA_{\cup_p},k}$
		};
		
		\node[draw, text width=4cm] (convtest) at ($(sol.south) + (0,-1.4)$) { 
			Is $\Delta x_{k} \leq \chi$ or $\Delta x_{l,k} \leq \chi$?
		};
		
		\node[] (yesconvtest) at ($(convtest.east) + (0.4,0.3)$) { 
			yes
		};
		
		\node[draw, text width=1.25cm] (conv) at ($(convtest.east) + (1.5,0)$) { 
			KKT point $v_l^*$ of $l$,\\
			$l$++, $\rho_{\max,l}=\rho$ 
		};

		\node[] (noconvtest) at ($(convtest.south) + (-0.3,-0.3)$) { 
			no
		};
		
		\node[draw, text width=4cm] (eval) at ($(convtest.south) + (0,-1.3)$) { 
			Evaluate $\Vert f_l^+(x_k + \Delta x_k)\Vert_2$ and $h_{\cup l-1}(x_k+\Delta x_k)$
		};
		
		\node[draw, text width=4cm] (test) at ($(eval.west) + (-3.5,0)$) { 
			Is the new point \\($h_{\cup l-1}$, $\Vert f_l^+\Vert_2$) acceptable to the filter $\mathcal{F}_l^k$?
		};
		\node[] (yestest) at ($(test.west) + (-0.4,0.3)$) { 
			yes
		};
		\node[] (notest) at ($(test.north) + (-0.3,0.3)$) { 
			no
		};
		
		\node[draw, text width=4cm] (trdec) at ($(test.north) + (0,1.74)$) { 
			\vspace{-17pt}
			\begin{equation}
				\rho = \rho/2\nonumber
			\end{equation}
		};
		
		\node[draw, text width=4cm] (modtest) at ($(test.west) + (-3.5,0)$) { 
			Is $\Delta f_l < \sigma \Delta q_l$ and $\Delta q_l > 0$?
		};
		\node[] (yesmodtest) at ($(modtest.north) + (0.8,0.2)$) { 
			yes
		};
		\node[] (nomodtest) at ($(modtest.north) + (-0.3,0.3)$) { 
			no
		};
		
		\node[draw, text width=4cm, align=left] (htype) at ($(modtest.north) + (0,1.55)$) { 
			If $\Delta q_l \leq 0$ then include ($h_{\cup l-1}$, $\Vert f_l^+\Vert_2$) in the filter (h-type iteration)
		};
		
		\node[draw, text width=4cm] (inc) at ($(htype.north) + (0,1.1775)$) { 
			$x_{k+1} = x_k + \Delta x_k$, $k$++, $\rho = \min(\rho_{\max,l},2\rho)$
		};

		
		\coordinate (b1)   at ($(Lin.north) + (-8,0.25)$);
		\coordinate (b2)   at ($(Lin.north) + (9,0.25)$);
		\coordinate (b3)   at ($(Lin.south) + (9,-0.25)$);
		\coordinate (b4)   at ($(Lin.south) + (-8,-0.25)$);
		
		\coordinate (c1)   at ($(sol.west) + (-3.5,0)$);
		\coordinate (c11)   at ($(sol.west) + (-3.5,0.)$);
		\coordinate (c2)   at ($(convtest.east) + (0.5,0)$);
		\coordinate (c3)   at ($(InitFilt.east) + (1.5,0)$);
		
		\draw[] (b1) -- (b2) -- (b3) -- (b4) -- (b1);
		\draw[decoration={markings,mark=at position 1 with
			{\arrow[scale=2,>=latex]{>}}},postaction={decorate}] ($(NLHLSP.east) + (-0.15,0)$) -- (Lin.west);
		\draw[decoration={markings,mark=at position 1 with
			{\arrow[scale=2,>=latex]{>}}},postaction={decorate}] (Lin.east) -- ($(HLSP.west) + (0.15,0)$);
		\draw[decoration={markings,mark=at position 1 with
			{\arrow[scale=2,>=latex]{>}}},postaction={decorate}] ($(HLSP.west) - (-0.15,0.5)$) -- ($(Lin.east) - (0,0.5)$);
		\draw[decoration={markings,mark=at position 1 with
			{\arrow[scale=2,>=latex]{>}}},postaction={decorate}] ($(Lin.south) - (-0.,0.25)$) -- ($(Init.north) - (0,0)$);
		\draw[decoration={markings,mark=at position 1 with
			{\arrow[scale=2,>=latex]{>}}},postaction={decorate}] ($(Init.east) - (-0.,0)$) -- ($(InitFilt.west) - (0,0)$);
		\draw[decoration={markings,mark=at position 1 with
			{\arrow[scale=2,>=latex]{>}}},postaction={decorate}] ($(InitFilt.south) - (-0.,0)$) -- ($(sol.north) - (0,0)$);
		\draw[decoration={markings,mark=at position 1 with
			{\arrow[scale=2,>=latex]{>}}},postaction={decorate}] ($(sol.south) - (-0.,0)$) -- ($(convtest.north) - (0,0)$);
		\draw[decoration={markings,mark=at position 1 with
			{\arrow[scale=2,>=latex]{>}}},postaction={decorate}] ($(convtest.south) - (-0.,0)$) -- ($(eval.north) - (0,0)$);
		\draw[decoration={markings,mark=at position 1 with
			{\arrow[scale=2,>=latex]{>}}},postaction={decorate}] ($(eval.west) - (-0.,0)$) -- ($(test.east) - (0,0)$);
		\draw[decoration={markings,mark=at position 1 with
			{\arrow[scale=2,>=latex]{>}}},postaction={decorate}] ($(test.west) - (-0.,0)$) -- ($(modtest.east) - (0,0)$);
		\draw[decoration={markings,mark=at position 1 with
			{\arrow[scale=2,>=latex]{>}}},postaction={decorate}] ($(modtest.north) - (-0.,0)$) -- ($(htype.south) - (0,0)$);
		\draw[decoration={markings,mark=at position 1 with
			{\arrow[scale=2,>=latex]{>}}},postaction={decorate}] ($(htype.north) - (-0.,0)$) -- ($(inc.south) - (0,0)$);
		\draw[decoration={markings,mark=at position 1 with
			{\arrow[scale=2,>=latex]{>}}},postaction={decorate}] ($(test.north) - (-0.,0)$) -- ($(trdec.south) - (0,0)$);
		\draw[] ($(inc.east) - (-0.,0)$) -- ($(c11.west) - (0,0)$) -- ($(c1.west) - (0,0)$);
		\draw[] ($(trdec.north) - (-0.,0)$) -- ($(c1.south) - (0,0)$);
		\draw[decoration={markings,mark=at position 1 with
			{\arrow[scale=2,>=latex]{>}}},postaction={decorate}] ($(c1.east) - (-0.,0)$) -- ($(sol.west) - (0,0)$);
		\draw[decoration={markings,mark=at position 1 with
			{\arrow[scale=2,>=latex]{>}}},postaction={decorate}] ($(convtest.east) - (0,0)$) -- ($(conv.west)$);
		\draw[decoration={markings,mark=at position 1 with
			{\arrow[scale=2,>=latex]{>}}},postaction={decorate}] ($(conv.north) - (0,0)$) -- ($(c3.south)$) -- ($(InitFilt.east)$);
		\draw[decoration={markings,mark=at position 1 with
			{\arrow[scale=2,>=latex]{>}}},postaction={decorate}] ($(modtest.north) + (0.5,0)$) -- ($(modtest.north) + (0.5,0.5)$) -- ($(modtest.north) + (3,0.5)$) -- ($(modtest.north) + (3,2.005)$) -- ($(trdec.west)$);
	\end{tikzpicture}}
	\caption{A symbolic overview of the sequential hierarchical least-squares programming (S-HLSP) with trust region and hierarchical step-filter based on the SQP step-filter~\citep{fletcher2002b} to solve non-linear hierarchical least-squares programs~\eqref{eq:NL-HLSP} with $p$ levels.
	}
	\label{fig:hfilter}
\end{figure*}

\section{Sequential hierarchical least-squares programming with trust-region and hierarchical step-filter}

\label{sec:shlsp}

Sequential hierarchical least-squares programming is commonly applied in real-time robot control~\citep{Kanoun2009, escande2014}. However, in these works no statements with regards to convergence properties are provided. The authors in~\cite{pfeiffer2018, pfeiffer2023} contextualized the control aspect within optimization and specifically introduced a trust-region constraint within the control framework. However, due to real-time requirements, the trust-region adaptation remained rudimentary (since a solution can not be recomputed with a different trust region radius due to computational limitations) and primarily focused on suppressing numerical instabilities in the case of kinematic and algorithmic singularities. In this work, we aim to embed S-HLSP fully within NLP. We formulate a hierarchical step-filter as described below to adjust the trust-region and determine whether a step from the HLSP sub-problem is acceptable to the original~\ref{eq:NL-HLSP}. 

\subsection{The SQP step-filter for a single level of the NL-HLSP}

The SQP step-filter (SQP-SF)~\citep{fletcher2002b} is a method for non-linear optimizers to measure the progress in the approximated sub-problems with respect to the original non-linear one. It thereby weighs progress in the objective function and feasibility of the constraints against each other without the necessity of tuning a penalty parameter. 

The filter applied to a level $l$ of the~\ref{eq:NL-HLSP} is a list of pairs $(h_{\cup l-1}, \Vert f_l^+\Vert_2)$ with
\begin{equation}
	h_{\cup l-1}(x_k+\Delta x_k) = \Vert f_{\bI_{\cup l-1}}^+ - v_{\bI_{\cup l-1}}^*\Vert_1 + \Vert f_{\bE_{\cup l-1}} - v_{\bE_{\cup l-1}}^*\Vert_1
\end{equation} 
The initialization with the filter point $(u,-\infty)$ limits constraint infeasibility $h_{\cup l-1} \leq u$.  

The HLSP sub-problem is solved for $\Delta x_k$, $\hat{v}_{\cup p,k}^*$ and the Lagrange multipliers $\hat{\lambda}_{\cup p}$ indicating the conflict between the hierarchy levels of the HLSP. 
If the step is sufficiently small $\Delta x_{k}\leq\chi$ or $\Delta x_{l,k}\leq\chi$, a KKT point $v_l^*$ of level $l$ has been identified. The step $\Delta x_{l}$ is defined in~\eqref{eq:stepdef} and is the sum of the primal sub-steps of levels $1$ to $l$ of the~\ref{eq:hlsp}. Using this convergence criterion captures cases where the current HSF level $l$ converges well but lower priority levels exhibit oscillations due to an inappropriately chosen trust region radius (see more in Sec.~\ref{sec:nstradapt}).

The new point $x_k+\Delta x_k$ is evaluated to $h_{\cup l-1}(x_k + \Delta x_k)$
and $\Vert f_l^+(x_k+\Delta x_k)\Vert_2$.
The new filter point $(h_{\cup l-1}, \Vert f_l^+\Vert_2)$ is then checked for acceptance with respect to the current filter. A step $\Delta x_k$ is acceptable to the filter if following condition holds for all $j\in\mathcal{F}_l^k$. $\mathcal{F}_l^k$ is the set of all points in the filter of level $l$ at the S-HLSP iteration $k$.
\begin{align}
	h_{\cup l-1} \leq \beta h_{\cup l-1}^j \qquad \text{or} \qquad \Vert f_l^+\Vert_2 + \gamma \leq \Vert f_l^{+j}\Vert_2
	\label{eq:acc}
\end{align}
$\beta$ is a value close to 1 and $\gamma$ is a value close to zero. Both values fulfill the condition $0 < \gamma < \beta < 1$.
Any point $(h_{\cup l-1}^j, \Vert f_l^{+j}\Vert_2)$ of the current filter $\mathcal{F}^k$
\textit{dominated} by the new point
\begin{align}
	h_{\cup l-1} \leq h_{\cup l-1}^j \qquad \text{and} \qquad \Vert f_l^+\Vert_2 \leq \Vert f_l^{+j}\Vert_2
\end{align}
is removed from the filter. Since the model reliably represents the non-linear problem, the trust region radius is increased. If this is not the case, the trust region radius is reduced and the HLSP is solved again, making ground for a new step $\Delta x_k$ that is again tested for filter acceptance. 
If the step is acceptable to the filter, the sufficient reduction criterion is checked
\begin{equation}
	\Delta f_l \geq \sigma \Delta q_l 
\end{equation} 
for positive $\Delta q_l$.
If this is not the case, the trust region is reduced and a new step $\Delta x_k$ is computed. Otherwise, the step is accepted. Only if $\Delta q_l \leq 0$, the new filter point is included in the filter (h-type iteration). 

\subsection{The hierarchical step filter for NL-HLSP}

\begin{algorithm}[t!]
	\caption{S-HLSP with HSF}\label{alg:hsf}
	\begin{algorithmic}[1]
		\Statex \textbf{Input:} {$x_0$, $\rho_0\gg 1$}
		\Statex \textbf{Output:} {$x$, $\lambda_{\cup p}$}
		\State $\rho = \rho_0$
		\State $k=1$
		\State $x_k = x_0$
		\For{$l=1,\dots,p$}
		\While{$\Vert \Delta x_{k} \Vert^2_2 > \chi$ and $\Vert \Delta x_{l,k} \Vert^2_2 > \chi$}
		\State $\Delta x_k,\hat{v}^*_{\cup p,k},\hat{\lambda}_{\cup l,k+1}\leftarrow$ Alg.~\ref{alg:solvehlsp} with $x_k$, $\rho$
		\State $\rho, acc \leftarrow$ SQP-SF on~\ref{eq:NL-HLSP} with $l$, $x_k$, $\Delta x_k$, $v_{\cup l-1}^*$, $\rho$
		\IIf{$acc$} $x_{k+1} = x_k + \Delta x_k$
		\State $k = k+1$
		\EndWhile
		\State $\rho_{\max,l+1} = \rho$
		\State $v_l^* = v_{l,k-1} = \hat{v}^*_{l, k-1}$
		\EndFor
		\State \Return $x$, $\lambda_{\cup p}$
	\end{algorithmic}
\end{algorithm}	

The HSF for S-HLSP considers prioritization by consecutively applying the SQP-SF on each level $l=1,\dots,p$ of the~\ref{eq:NL-HLSP}. An algorithmic and symbolic overview of the S-HLSP with HSF is given in Alg.~\ref{alg:hsf} and Fig.~\ref{fig:hfilter}, respectively.
First, the trust region radius is set to $\rho = \rho_0$ and the maximum trust region radius to $\rho_{\max,0} \gg 1$.
The S-HLSP now iterates through the~\ref{eq:NL-HLSP} by applying the SQP-SF to the~\ref{eq:NL-HLSP} of level $l$ with maximum trust region radius $\rho \leq \rho_{\max,l-1}$. After the~\ref{eq:hlsp} sub-problem has been resolved, the SQP-SF checks whether the step $\Delta x_k$ is acceptable~\eqref{eq:acc} to the filter by evaluating the~\ref{eq:NL-HLSP} at both $x_k$ and $x_k + \Delta x_k$. The trust region radius is increased ($\rho\leftarrow \max(\rho_{\max,l},2\rho)$ in case of step acceptance and decreased ($\rho\leftarrow \rho/2$) otherwise.
Note that the initial trust region radius does not carry any compatibility requirement with regards to the inactive inequality constraints $\mI_{\cup l-1}$ as in the original filter. Due to the least-squares formulation with slacks in the~\ref{eq:hlsp} sub-problems, a compatible solution $v_{\cup p}^*$ can always be found. This discards the need for a restoration phase~\citep{fletcher2002b} in our algorithm.

The~\ref{eq:hlsp} sub-problems are solved repeatedly for steps $\Delta x_k$. If the step is acceptable to the SQP-SF, the step is applied  $x_{k+1} = x_k + \Delta _k$ until a KKT point of~\ref{eq:NL-HLSP} is reached with $\Vert \Delta x_{k} \Vert^2_2 \leq \chi$ or $\Vert \Delta x_{l,k} \Vert^2_2 \leq \chi$. Upon convergence, we set $\rho_{\max,l} = \rho$ and save $v_l^*$. Note that at this point the slacks from the non-linear and linearized problems $v_l^*$ and $\hat{v}_l^*$ coincide. The SQP-SF is then applied to the next level $l\rightarrow l+1$ of the~\ref{eq:NL-HLSP}.

In the case of step acceptance in the SQP-SF, the trust region radius is increased to a level below $\rho_{\max,l-1}$. Due to this restriction on the trust region radius, we can make following statement with regards to the convergence properties of the HSF.	
\begin{theorem}
	The hierarchical step-filter is globally convergent to the~\ref{eq:NL-HLSP} if for each level $l$ $\rho \leq \rho_{\max,l}$. 
\end{theorem}

\begin{proof}
	Assume that level~$l_K$ of~\ref{eq:NL-HLSP} converges to a KKT point $(x^*, v_{l_K}^*)$  at a given trust region radius $\rho_{\max,l_K}=\rho$ due to the usage of the globally convergent SQP-SF~\citep{fletcher2002b}. The condition $\rho \leq \rho_{\max,l_{\text{K}}}\leq \dots \leq \rho_{\max,1}$ ensures that the~\ref{eq:hlsp} sub-problem of level $l = l_{\text{K}}+1$ represents the~\ref{eq:NL-HLSP}'s of levels $1$ to $l_{\text{K}}$ sufficiently. Therefore, each new point $x_{k+1} = x_k + \Delta x_k$ resulting from a step $\Delta x_{k}$ of the~\ref{eq:hlsp} sub-problem of the step-filter of level $l$ is still a KKT point $v_{\cup l_{\text{K}}}^*$ of levels $1$ to $l_{\text{K}}$. Global convergence of the HSF follows from applying above successively to levels $l_K = 1,\dots,p-1$.
\end{proof}
In practice, we observed that we can reset the trust region radius to a larger value than $\rho_{\max,l}$ without disturbing KKT points $v_{\cup l-1}^*$ of higher priority levels. This leads to faster convergence as the resolution of lower priority levels is not restricted by a small trust region radius resulting from the bad resolution of a higher priority one, for example due to a bad Hessian approximation. As we show next, we can formulate a more straightforward global convergence proof for the HSF without a condition on the trust region radius after all.

\begin{theorem}
	\label{th:globconv2}
	The hierarchical step-filter is globally convergent to the~\ref{eq:NL-HLSP} as long as the filter of each level $l$ is initialized with a point $(u,-\infty)$ for all $u > 0$. 
\end{theorem}

\begin{proof}
	Assume that level~$l_K$ of~\ref{eq:NL-HLSP} converges to a KKT point $(x^*, v_{l_K}^*)$  at a given trust region radius $\rho_{\max,l_K}=\rho$ due to the usage of the globally convergent SQP-SF~\citep{fletcher2002b}. For the next level $l=l_K+1$ with initial $x = x^*$ and $\Delta x_k$ resulting from solving the~\ref{eq:hlsp}, we have $h_{\cup l_K} = \Vert f_{\cup l_K}^+(x + \Delta x) - v_{\cup l_K}^*\Vert_1 = 0$ only for linear constraints or if $\rho \leq \rho_{\max,_{l_K}}$. Otherwise, we necessarily have $h_{\cup l_K} = \Vert f_{\cup l_K}^+(x + \Delta x) - v_{\cup l_K}^*\Vert_1 > 0$. The filter with point $(u,-\infty)$ can therefore only accept a new point $(h_{\cup l_K}, \Vert f_l^+\Vert_2)$ if an infeasible point $h_{\cup l_K} \leq u$ with $u>0$ is admissible. If so, the global convergence proof of~\citep{fletcher2002b} applies to the SF of level $l$. Global convergence of the HSF follows from applying above successively to levels $l_K = 1,\dots,p-1$
\end{proof}

Throughout the remainder of the article, we use the specifications of theorem~\ref{th:globconv2} for the implementation of our algorithm (see Alg.~\ref{alg:nstra}).

\subsection{Some considerations towards a comprehensive hierarchical filter}

\subsubsection{Solving the whole HLSP}
One question that naturally arises is whether in every outer S-HLSP iteration the~\ref{eq:hlsp} should be solved fully up to and including level $p$ or only up to the current HSF level $l$. The latter option would presumably require less computational effort since only the relevant part of the HLSP is solved that is actually steered towards a KKT point by the HSF. However, we made the empirical observation that by doing so the S-HLSP requires overall more iterations (even without the nullspace trust-region adaptation, see~\ref{sec:nstradapt}). One explanation could be that even if the filter is treating level $l$, the lower levels still benefit from the step resulting from solving the whole HLSP. For example, if level $l$ has continuous optima like seen for the disk constraint (Sec.~\ref{eval:nlpt}), it might end up at a minimum that is less optimal for a lower priority level $l+1$. 
On the other hand, if the whole HLSP is solved some informed steps towards an optimum of $l+1$ may already have been achieved.

\subsubsection{Nullspace trust region adaptation}
\label{sec:nstradapt}

\begin{algorithm}[t!]
	\caption{S-HLSP with HSF and NSTRA}\label{alg:nstra}
	\begin{algorithmic}[1]
		\Statex \textbf{Input:} {$x_0$, $\rho_0\gg 1$}
		\Statex \textbf{Output:} {$x$, $\lambda_{\cup p}$}
		\State $k=1$
		\State SQP-SF = \{SQP-SF$_1$, \dots, SQP-SF$_p$\}
		\For{$l=1,\dots,p$}
		\State $\rho_{1:p} =\rho_{0,1:p}$
		\While{$\Vert \Delta x_{k} \Vert^2 > \chi$ and $\Vert \Delta x_{l,k} \Vert^2 > \chi$}
		\State $\Delta x_k,\hat{v}^*_{\cup p,k},\hat{\lambda}_{\cup p,k+1}\leftarrow$ Alg.~\ref{alg:solvehlsp} with $x_k$, $\rho_{1:p}$
		\For{$j=l,\dots,p$}
		\State $\rho_j, acc_j \leftarrow$ SQP-SF$_j$ on~\ref{eq:NL-HLSP} with $x_k$, $\Delta x_k$, $v_{\cup l-1}^*$, $\rho_j$
		\EndFor
		\IIf{$acc_l$} $x_{k+1} = x_k + \Delta k$
		\State $k = k+1$
		\EndWhile
		\State $v_l^* = v_{l,k-1} = \hat{v}^*_{l, k-1}$
		\EndFor
		\State \Return $x$, $\lambda_{\cup p}$
	\end{algorithmic}
\end{algorithm}	

Given the fact that we always solve the full~\ref{eq:hlsp} sub-problem including all priority levels, it would be desirable to let  priority levels lower than the current HSF level converge as much as possible. 
We achieve this by integration of a trust region adaptation method directly within the~\ref{eq:hlsp} sub-problem solver. This nullspace trust region adaptation (NSTRA) within S-HLSP according to theorem~\ref{th:globconv2} is described in Alg.~\ref{alg:solvehlsp} and Alg.~\ref{alg:nstra}.

We create an SQP-SF for each priority level $j=l,\dots,p$. The current trust region radii $\rho_{1:p}$ are passed to the HLSP solver~Alg.~\ref{alg:solvehlsp}.
Here, the trust region radius is consistently updated to $\rho = \rho_{l+1}$ after a level $l$ of the~\ref{eq:hlsp} has resolved. This allows us to obtain a comprehensive overview of `free' variables (which are not used by higher priority levels) by means of nullspace basis computation of the active set $\mathcal{A}_l$. For example, if we have $n=4$ and $f_1(x_1, x_2, x_3)\in\mathbb{R}^{2}$, $J_1(x_1, x_2, x_3, x_4)\in\mathbb{R}^{2\times 4}$ and $f_2(x_3, x_4)\in\mathbb{R}^{1}$, $J_2(x_1, x_2, x_3, x_4)\in\mathbb{R}^{1\times 4}$ we would assume that the free variable for level $2$ is $x_4$ (since $f_1$ occupies $x_1$, $x_2$ and $x_3$). However, at a closer look it becomes apparent that the Jacobian of $f_1$ can be of  at most rank $r_{\max}=\min(m_1,n)=m_1=2$. Therefore, there are at least two ($n-m_1=2$) free variables left for level 2 (more if $J_1$ is rank deficient). Note that we do not make an explicit analysis of the nullspace bases of each level to determine the free variables. Occupied variables are naturally removed by the projection of the inactive constraint $A_{\mI_{\cup l-1}}$ (which contain the inactive trust region constraints) into the nullspace of level $l$.

After the HLSP step computation $\Delta x$ and $\lambda_{\cup p}$ from Alg.~\ref{alg:solvehlsp}, the filters of all levels are updated. The overall step acceptance is determined by the filter of the current level $l$. 
At the same time, the trust region radii of lower level $l+1,\dots,p$ are modified according to the SQP-SF state of their respective levels.

This adaptation method does not change the dual variables of already active constraints since the corresponding primal variables are already fixed and the dual itself does not depend on the right hand side $b$ (which represents the trust region radius). 
Furthermore, this adaptation method has no indications for inactive constraints $A_{\mA_{\cup_{l-1}}}$ as long as the trust region radius is larger than zero (which is the case for $\rho_0>0$). Otherwise, an infeasible inactive constraint may arise which would lead to solver failure since the log function of the IPM log-barrier is not defined for negative values of the slack variables.

Conflicts can arise between overall accepted steps on the HSF level $l$ and potential step rejection on a lower priority level $i > l$. We introduce a filter inertia that delays trust region increases on lower levels. Only if a certain number of steps of level $i$ is accepted (without step rejections in-between), the trust region is increased. Otherwise, we observed oscillations in the trust region radius due to a repeated interchange between step acceptance and rejection on level $i$. 

We observed that too stark differences in the trust region radii lead to numerical instability, especially in the case of BFGS Hessian approximations. This may be explained by the direct dependency on $\Delta x$ whose `conditioning' (difference between maximum and minimum absolute value) is directly influenced by the variable wise trust region adaptation. We therefore adapt the trust region radii of level $j = l+1,\dots,p$ in the following way ($l$ is the level the HSF is currently working on)
\begin{equation}
	\rho_j \leftarrow \max(\rho_{l}/\kappa, \max(\chi^2, \rho_j / 2)) \quad\text{for}\quad j = l+1,\dots,p
\end{equation}
$\rho_{l}$ is the trust region radius of the HSF level $l$. $\kappa = 100$ is a stability factor. $\chi$ is the step threshold. For good measure we also do not increase the trust region radii of lower levels above the trust region radius $\rho_l$ of the current HSF level $l$.

The nullspace trust region adaptation method has no indications for the  convergence property of the HSF since it takes place in the nullspace of the current HSF level $l$.

\subsubsection{Removing primal sub-steps of rejected levels}

Another possibility is to remove primal sub-steps of $\Delta x$ that correspond to rejected priority levels lower than the current (accepted) HSF level $l$. If the filter of a lower priority level $j>l$ determines that the current step is not acceptable, the primal sub-component $\Delta x^j$ is removed from $\Delta x = \sum_{l=0}^{p} \Delta x^l$. However, if a level $j$ is removed from $\Delta x$, we also need to remove the other sub-steps $j+1$ to $p$ since they would not be true primal nullspace steps of levels $1$ to $j$ anymore. We did not observe any improvement in iteration numbers of the HSF and therefore do not further address this method throughout the remainder of the article.

\section{Sparse Hierarchical Least-squares programming}
\label{sec:turnback}

While above we have outlined our HSF to find local solutions of~\ref{eq:NL-HLSP}'s, we have not further addressed the resolution of the~\ref{eq:hlsp} sub-problems.
Our previously reported solver $\mathcal{N}$IPM-HLSP~\citep{pfeiffer2021} does not specifically take sparsity of the underlying HLSP into account. This is problematic in the case of discrete optimal control since the computational complexity of resolving the HLSP would increase cubically with the control horizon~\citep{wangboyd2010}. 
We therefore propose a sparse hierarchical IPM solver based on the reduced Hessian formulation which we refer to as s-$\mathcal{N}$IPM-HLSP. Critically, it preserves the bands of the constraint matrices $A$ during the computation of the nullspace basis by means of the turnback algorithm~\citep{berry1985} and therefore maintains linear complexity in the control horizon length.
The turnback algorithm is based on the observation that the individual columns of the nullspace basis $Z$ of a matrix $A$ (with $AZ=0$) need not to be null-vectors with respect to all columns of $A$ but only a linearly independent subset of it. The main algorithmic step is therefore to identify these subsets by augmenting columns of $A$ into a sub-matrix until a linearly dependent column is added (and at which point the linearly independent column subset has been identified). It can be observed that these subsets often overlap. Our implementation of the turnback algorithm takes advantage of this circumstance by reusing previously computed subsets. While this is only a detail within the turnback algorithm,  to the best of our knowledge we have not seen it being documented in previous works (along with further implementation details that we give here). This leads to significant computational speedup (around 100\% depending on the problem constellation).

\subsection{A recycling implementation of the turnback algorithm}

The turnback algorithm~\citep{topcu1979} first computes a rank-revealing decomposition of the matrix $A$ in order to determine the rank of the nullspace basis. We rely on the rank revealing LU decomposition (a turnback algorithm based on the QR decomposition has been proposed~\citep{kaneko1982}) given as follows
\begin{equation}
	A =P^TLU Q^T=  P^TL\BIN U_1 & U_2 \\ 0_{(m-r) \times r} & 0_{(m-r) \times (n-r)}\BOUT Q^T
\end{equation}
$r$ is the rank of the matrix $A$.
$P\in\mathbb{R}^{m\times m}$ and $Q\in\mathbb{R}^{n\times n}$ are permutation matrices. $L\in\mathbb{R}^{m\times m}$ is a full-rank lower triangular matrix. $U\in\mathbb{R}^{m\times m}$ is composed of the
upper triangular matrix $U_1\in\mathbb{R}^{r\times r}$ and the regular matrix $U_2\in\mathbb{R}^{r\times n-r}$ which may expose some bands if the original matrix $A$ is structured. 
A basis of the nullspace of the matrix $A$ can be constructed as follows
\begin{equation}
	Z = Q  \BIN -U_1^{-1}U_2 \\ I \BOUT
	\label{eq:luns}
\end{equation}
If $A$ is structured and is banded, this is reflected in the matrix $U_2$. The sparsity of the matrix $Z$ is thereby influenced by $U_2$ but not fully determined as described in~\citep{berry1985}. A simple example would be 
\begin{align}
	Z = Q 
	\BIN  -\BIN 1 & 1  \\
	0 & 1 \BOUT 	&	\BIN 1 & 0  \\
	0 & 1 \BOUT \\ 1 & 0 \\ 0 & 1\BOUT = Q\BIN -1 & -1 \\ 0 & -1 \\ 1 & 0 \\ 0 & 1\BOUT
\end{align}
which does not preserve the banded structure of $U_2 = \BIN 1  & 0 \\ 0 & 1 \BOUT$. We therefore rely on the algorithm described in Alg.~\ref{alg:findb} to determine the index vector $b\in\mathbb{R}^{n-r}$ which indicates the first  or last  non-zero entry of each column of $Z$. The former case refers to the right-looking and the latter one to the left-looking turnback algorithm (which we refer to as `direction' ($d$) in the following).

\begin{algorithm}[t!]
	\caption{Routine for determining index vectors $b$ and $p_v$}\label{alg:findb}
	\begin{algorithmic}[1]
		\Statex \textbf{Input:} {$A\in\mathbb{R}^{m\times n}$}
		\Statex \textbf{Output:} $b$, $p_v$, $d$, $r$
		\State $r,P,L,U,Q\leftarrow {\tt LU}(A)$
		\State $b_{\text{right}} \leftarrow$  Rows of first $nnz$ in each column of $Q\BIN U_2^T & I\BOUT^T$
		\If{double entries in $b_{\text{right}}$}
		\State $b_{\text{left}}\leftarrow$  Rows of last $nnz$ in each column of $Q\BIN U_2^T & I\BOUT^T$	
		\EndIf
		\State $b,d\leftarrow$ minimum number of double entries ($b_{\text{right}}$, $b_{\text{left}}$)
		\State $p_v \leftarrow$  Rows of $I$ of $Q\BIN U_2^T & I\BOUT^T$
		\If{$d$ = `right'}
		\State ${\tt sortAscending}$($b$)
		\State $b[0] = 0$
		\ElsIf{$d$ = `left'}
		\State ${\tt sortDescending}$($b$)
		\State $b[0] = n-1$			
		\EndIf
		\State Arrange $b$ such that $p_{v,i} \in \BIN b_i, b_{i+1}\BOUT$ for $i=1,\dots,n-r$
		\State \Return $b$, $p_v$, $d$, $r$
	\end{algorithmic}
\end{algorithm}	

First we compute the rank-revealing LU decomposition of $A$ with rank $r$ (routine ${\tt LU}$).
Using the permutation $Q \BIN U_2^T & I\BOUT^T$, we obtain a rough estimate of $b$ (as without  evaluating $U_1^{-1}\BIN U_2^T & I\BOUT^T$ we can not fully know the sparsity structure of $Z$). Furthermore, this determines the pivot column vector $p_v$ (the columns of $A$ that correspond to $I$ of $\BIN U_2^T & I\BOUT^T$).
The vector $b$ is then checked for double entries. If there are double entries, the direction is switched and $b$ is recomputed.
Afterwards, the direction with the least number of double entries is chosen. The vector $b$ is sorted from smallest to largest for the right-looking algorithm by ${\tt sortAscending}$ (for left-looking, from largest to smallest by ${\tt sortDescending}$). If not already so, the first entry of $b$ is set to zero for the right-looking algorithm (for left-looking, set the first entry to $n-1$).
Finally, $b$ is arranged (possible changing values) such that the pivot columns are contained within their own column brackets ($p_{v,i} \in \BIN b_i, b_{i+1}\BOUT$ where $i=1,\dots,r$ indicates the $i$-th entry of $p_v$ or $b$).

We now proceed with the actual turnback algorithm by looping through the ordered vector $b$ with the current index $b_i$ ($i=1,\dots,n-r$). The algorithmic details are given in Alg.~\ref{alg:recturnback}. The symbol $\pm$ makes the distinction between right ($+$) and left looking case ($-$).

\begin{algorithm}[t!]
	\caption{Recycling turnback algorithm}\label{alg:recturnback}
	\begin{algorithmic}[1]
		\Statex \textbf{Input:} $A\in\mathbb{R}^{m\times n}$
		\Statex \textbf{Output:} $Z$
		\State $b,p_v,d,r\leftarrow$ {Alg.}~\ref{alg:findb} with $A$
		\State $Z \in \mathbb{R}^{n\times n-r}$
		\State $b_0 = 0$
		\State $j = b_1$
		\For{$i=1,\dots,n-r$}
		\State $r_{\max} = {\tt getMaxR}(A,b_i,d)$
		\State $c=b_i \pm\max(r_{\max}, \vert b_{i+1} - b_i\vert)$
		\State $o = (\vert c\pm b_i\vert) / (\vert j - b_{i-1}\vert)$
		\If{$i=1$ or $o < \hat{o}$} 
		\State $r, P,L,U,Q = {\tt LU}(A^{b_{i}:c})$
		\Else 
		\State $r, P,L,U,Q = {\tt updateLU}(A^{b_{i}:c})$
		\EndIf
		\State $j = c$, linDep = $\text{False}$
		\While{$1 \leq j \leq n$}
		\If{linDep}
		\If{${p_{v,i}}\notin U_2$}
		\State $\hat{r} = r$
		\State $r = {\tt remColLU}(p_{v,i},A)$
		\If{$r = \hat{r}$} 
		\State $r,P,L,U,Q = {\tt addColLU}(p_{v,i},A)$
		\Else
		\State $P,L,U,Q = {\tt addColLU}^*(p_{v,i},A)$
		\EndIf
		\Else
		\State $z = U_1^{-1}U_{2,p_{v,i}}$
		\State $Z_i = \BIN z^T & 0\BOUT$, $Z_{i,p_{v,i}} = 1$, $Z_i = QZ_i$
		\State break
		\EndIf
		\EndIf
		\While{$j\in p_v$}
		\State $j\leftarrow j\pm1$
		\EndWhile
		\State $\hat{r} = r$
		\State  $r,P,L,U,Q = {\tt addColLU}(A,j)$
		\IIf{$r = \hat{r}$} linDep = $\text{True}$
		\EndWhile
		\EndFor
		\State \Return $Z$
	\end{algorithmic}
\end{algorithm}

First, the maximally achievable rank $r_{\max}$ is computed by the routine ${\tt getMaxR}$. This searches for pivot rows in the columns of $A$ beginning from the current column $b_i$ going to the right or left and gives a rough estimate on how many columns are required for linear dependency. A pivot row is valid if it has not been chosen before.
With this information the factorization of $A^{b_{i}:c}$ from the current column $b_i$ to column $c=b_i\pm\max(r_m, b_{i+1} - b_i)$ is computed. The term $b_{i+1}-b_i$ ensures that the current pivot column $p_{v,i}$ is contained within the current column subset. If there is sufficient overlap  ($o \geq \hat{o}$ with the constant $\hat{o}=0.5$) between the previous factorization of $A^{b_{i-1}:j}$ and the desired one $A^{b_{i}:j=b_{i}+c}$, a series of column removals (on the left / right of $b_i$; routine ${\tt remColLU}$) and additions (on the right / left of $j$; routine ${\tt addColLU}$) is conducted (${\tt updateLU}$). Note that for an unordered $b$ the probability of overlap is reduced or non-existent since the two columns $b_i$ and $b_{i-1}$ might be too far away from each other.
The algorithm then keeps adding non-pivot columns on the right / left of $j$ (and consequently increases / decreases $j$)  until a linear dependent column is found (column addition does not lead to rank increase).
If the permuted pivot column $p_{v,i}$ is already contained within $U_2$ ($p_{v,i}\in U_2$),
the pivot column $p_{v,i}$ is removed from the decomposition.
If the rank stays the same ($r=\hat{r}$), re-add the column in a non-permuting fashion such that it ends up in $U_2$.
If the rank decreases, re-add the column in the normal permuting fashion and keep adding further columns $j$.
Once the pivot column is contained within $U_2$ as $u_2\coloneqq U_{2,p_{v,i}}$, the null-vector $z_1 = U_1^{-1} u_2$ is computed. We set the $i$-th (or its corresponding permutation resulting from sorting $b$) column  of $Z_{i} = \BIN z_1^T & 0\BOUT^T$ and the $p_{v,i}$-th entry of the column to $Z_{i, p_{v,i}}=1$. Finally, we apply the permutation $Z_{i} = QZ_{i}$.

\subsubsection{Full rank requirement}
The resulting nullspace basis is full-rank due to singular usage of the pivot-columns during the column augmentation~\citep{berry1985}. Pivot column $i$ is only used once during the turnback process for the current nullspace basis column $b_i$, and its corresponding $p_{v,i}$-th entry is set to 1. This means that the resulting nullspace basis has the structure
$Z = \hat{Q}  \BIN T^T & I \BOUT^T$
which is clearly full column rank. $\hat{Q}$ is some permutation matrix and $T$ represents the turnback nullspace columns $z$.

\subsubsection{Conditioning of the nullspace basis}
For numerical stability, it is desirable that the resulting nullspace basis is well conditioned. We incorporate an approximate condition number maintenance by measuring the smallest to the largest absolute value of the factor $U_2$ (which needs to be inverted). This leads to continued column augmentation to the current sub-matrix even if linear dependency has already been detected (but with high approximate condition number of $U_2$).

\subsubsection{Growth of $L$}
One disadvantage of conducting a large number of updates is that the factor $L$ representing the Gaussian elimination steps of the updates grows linearly with the number of updates (with some pre-factor depending on the magnitude of the resulting spike of inserting or removing a column). LUSOL~\citep{gill1987}  (the library we use to conduct the rank-revealing LU decomposition) does not consider the sparsity of the new column $c_{\text{new}}$ to be added. During the forward substitution step $L \hat{c} = c_{\text{new}}$ of the Barthels-Golub update~\citep{gill1987}, we avoid zero operations by only considering the rows of $L$ which correspond to rows below and including the first non-zero entry of $c_{\text{new}}$. Another possibility would be to regularly recompute the LU decomposition of the current sub-matrix to dispose of excessively large factors $L$. In future work we would like to implement truly sparse LU column updates by using graph theory, for example as described in~\citep{demmel1995}.

\section{Evaluation}
\label{sec:eval}

We implemented both S-HLSP with trust region and HSF and s-$\mathcal{N}$IPM-HLSP in C++ based on the Eigen library~\citep{eigenweb}. The recycling turnback algorithm is based on the rank-revealing sparse LU factorization Fortran library LUSOL~\citep{gill1987} with C~interface (https://github.com/nwh/lusol). All column additions and deletions are conducted by the column replacement routine ${\tt lu8rpc}$, maintaining one factorization of the dimensions of the matrix to be decomposed. The simulations are run on an 11th Gen Intel Core i7-11800H \@ 2.30GHz $\times$ 16 with 23GB RAM.

We first evaluate the computational efficiency of the recycling turnback algorithm, see Sec.~\ref{eval:turnback}. Secondly, the convergence property of the S-HLSP is tested on non-convex~\ref{eq:NL-HLSP}'s with optimization test functions (see Sec.~\ref{eval:nlpt}), inverse kinematics of the humanoid robot HRP-2 (Sec.~\ref{eval:ik}) and prioritized discrete optimal control problems for the robot dog Solo12~\citep{grimminger2020} (Sec.~\ref{eval:trajopt}). We compare our HLSP solvers s-$\mathcal{N}$IPM-HLSP (turnback algorithm) and s-$\mathcal{N}$IPM-HLSP (LU-NS) (\eqref{eq:luns} with LU decomposition) based on the reduced Hessian formulation to the hierarchical versions of H-GUROBI~\citep{gurobi}, H-MOSEK~\citep{mosek}, H-PIQP~\citep{piqp} and H-OSQP~\citep{osqp}. All these solvers are sparse and resolve the~\ref{eq:hlsp} directly. The `H-' indicates that the computation times not only include the solver times but also the active and inactive set assembly. Thereby, all solvers including s-$\mathcal{N}$IPM-HLSP are based on the same hierarchical solver framework with comparable computational complexity. Note that we do not evaluate tailored QP solvers for discrete optimal control based on the Riccati recursion~\citep{hpipm}, since they only allow the use of `forward' dynamics (explicit computation of states, which is less computationally efficient than `inverse' dynamics for the explicit computation of controls~\citep{carpentier2018}). In the following, the number of non-zeros ($nnz$) refers to the sum of non-zeros of the constraint matrices $A_{\mathbb{E}_l}$, $A_{\mathbb{I}_l}$, $A_{\mA_{l-1}}$ and $A_{\mI_{l-1}}$ or its projected counterparts $A_lN_{l-1}$ (with $nnz(A_{\mA_{l-1}}N_{l-1})=0$) for s-$\mathcal{N}$IPM-HLSP. This number may differ from the actual number of non-zeros that are handled by the respective solvers, for example during the factorization of the IPM normal form. For all examples, the reported computation times only refer to the required time of resolving the~\ref{eq:hlsp} sub-problems. Computation times for example for data allocation / setting or Jacobian and Hessian computations are excluded since they either depend on the computational efficiency of external libraries and are not the focus of this work.

In all examples, the trust region constraint is defined on $l=0$.
The step and Hessian augmentation thresholds $\chi$ and $\epsilon$ are hand-picked for each problem in order to achieve fast convergence. We therefore see a necessity to further investigate their influence on the accuracy and optimality of the solutions and to design automatic selection methods.

\subsection{A recycling Turnback LU for sparse HLSP}
\label{eval:turnback}

\begin{table*}[htp!]
	\centering
	\resizebox{1.\columnwidth}{!}	{%
		\begin{tabular}{@{} cccccccccccccccc @{}}  
			\toprule
			& & & & & & & \multicolumn{5}{c}{Turnback-NS} & \multicolumn{3}{c}{LU-NS} & LU\\ 			
			\cmidrule(lr){7-12} \cmidrule(lr){13-15} \cmidrule(lr){16-16}
			$n_x$ & $n_u$ & $T$ & $n$ & $nnz(A^TA)$ & $d(A^TA)$ &$nnz(Z^TZ)$  & $d(Z^TZ)$ & col+ & col- &  $t$ [ms] & $t^*$ [ms] &  $nnz(Z^TZ)$ & $d(Z^TZ)$ &$t$ [ms] & $t$ [ms]\\
			\cmidrule(lr){1-6}\cmidrule(lr){7-12} \cmidrule(lr){13-15} \cmidrule(lr){16-16}
			12 & 3 & 10 & 150 & 5358 & 0.24 & 594 & 0.66 & 360 & 197 & 1.08 & 1.94 & 900 & 1 & 0.48 & 0.13\\
			12 & 3 & 20 & 300 & 6996 & 0.22 & 1344 & 0.37 & 480 & 287 & 1.74 & 4.67 & 3600 & 1 & 1.6 & 0.31 \\
			12 & 3 & 30 & 450 & 8814 & 0.20 & 2094 & 0.26 & 600 & 377 & 2.19 & 7.30 & 8100 & 1 & 1.68 & 0.17\\
			12 & 3 & 40 & 500 & 22908 & 0.06 & 2844 & 0.20 & 1440 & 827 & 4.81 & 10.55 & 14400 & 1 & 7.44 & 0.33 \\
			12 & 3 & 50 & 650 & 28758 & 0.05 & 3594 & 0.16 & 1800 & 1037 & 6.55 & 13.61 & 22500 & 1 & 7.54 & 0.42 \\
			\cmidrule(lr){1-6}\cmidrule(lr){7-12} \cmidrule(lr){13-15} \cmidrule(lr){16-16}
			12 & 6 & 10 & 180 & 6996 & 0.22 & 1344 & 0.37 & 480 & 287 & 1.63 & 3.79 & 3600 & 1 & 0.70 & 0.13\\
			12 & 6 & 20 & 360 & 14556 & 0.11 & 2844 & 0.20 & 960 & 587 & 3.45 & 6.74 & 14400 & 1 & 1.91 & 0.27\\
			12 & 6 & 30 & 540 & 22116 & 0.08 & 4344 & 0.13 & 1440 & 887 & 6.13 & 10.40 & 32400 & 1 & 5.34 & 0.35\\
			12 & 6 & 40 & 720 & 29676 & 0.06 & 5844 & 0.10 & 1920 & 1187 & 8.73 & 15.00 & 57600 & 1 & 15.07 & 0.38\\
			\cmidrule(lr){1-6}\cmidrule(lr){7-12} \cmidrule(lr){13-15} \cmidrule(lr){16-16}
			12 & 9 & 10 & 210 & 8814 & 0.20 & 2094 & 0.26 & 600 & 377 & 1.80 & 3.45 & 8100 & 1 & 0.80 & 0.16\\
			12 & 9 & 20 & 420 & 18264 & 0.10 & 4344 & 0.13 & 1200 & 767 & 4.77 & 8.00 & 32400 & 1 & 2.84 & 0.26\\
			12 & 9 & 30 & 630 & 27714 & 0.07 & 6594 & 0.09 & 1800 & 1157 & 7.83 & 13.30 & 72900 & 1 & 7.78 & 0.38\\
			\cmidrule(lr){1-6}\cmidrule(lr){7-12} \cmidrule(lr){13-15} \cmidrule(lr){16-16}
			12 & 12 & 10 & 240 & 10812 & 0.19 & 2844 & 0.20 & 720 & 467 & 2.36 & 5.79 & 14400 & 1 & 0.95 & 0.19\\
			12 & 12 & 20 & 480 & 22332 & 0.10 & 5844 & 0.10 & 1440 & 947 & 5.05 & 9.89 & 57600 & 1 & 4.10 & 0.29\\
			12 & 12 & 30 & 720 & 33852 & 0.07 & 8844 & 0.07 & 2160 & 1427 & 9.75 & 17.31 & 129600 & 1 & 9.47 & 0.40\\
			\cmidrule(lr){1-6}\cmidrule(lr){7-12} \cmidrule(lr){13-15} \cmidrule(lr){16-16}
			12 & 15 & 10 & 270 & 12990 & 0.18 & 3596 & 0.16 & 840 & 557 & 3.70 & 4.90 & 22500 & 1 & 1.50 & 0.17\\
			12 & 15 & 20 & 540 & 26760 & 0.09 & 7346 & 0.08 & 1680 & 1127 & 7.50 & 11.80 & 90000 & 1 & 7.36 & 0.53\\
			\cmidrule(lr){1-6}\cmidrule(lr){7-12} \cmidrule(lr){13-15} \cmidrule(lr){16-16}
			12 & 18 & 10 & 300 & 15348 & 0.17 & 4344 & 0.13 & 960 & 647 & 2.97 & 5.73 & 32400 & 1 & 1.49 & 0.19\\
			12 & 18 & 20 & 600 & 31548 & 0.09 & 8844 & 0.07 & 1920 & 1307 & 8.64 & 14.32 & 129600 & 1 & 8.91 & 0.34\\
			\bottomrule
		\end{tabular}
	}
	\caption{Results for the turnback LU algorithm for the computation of sparse nullspace bases of banded matrices. $nnz$ are the number of non-zeros of a matrix. $d$ is the density of a matrix. col+ and col- are the number of column additions and removals, respectively. $t^*$ is the computation time for the non-recycling turnback algorithm.}
	\label{tab:turnback}
\end{table*}

We apply the turnback algorithm to matrices with band structure of the following form
\begin{align}
	&A\in\mathbb{R}^{Tn_s\times T(n_s+n_c)} =\\
	&\BIN 
	C\in\mathbb{R}^{n_x\times n_u} & I & \cdots & 0 & 0 & 0\\
	0 & S\in\mathbb{R}^{n_x\times n_x} & \cdots & 0 & 0 & 0 \\
	\vdots & \vdots & \ddots & \vdots & \vdots & \vdots \\
	0 &  0 & \cdots & S & C & I\\
	\BOUT\nonumber
\end{align}
This is a common structure of constraint matrices $A$ resulting from the linearization of system dynamics of the form $s_{k+1} = f(s_k, c_{k})$ with $k=1,\dots ,T$. The overall variable vector is $x = \BIN c_1^T & s_2^T & \dots & c_{T-1}^T & s_T^T \BOUT^T$. $S$ is the Jacobian of $f$ with respect to the state vector $s\in\mathbb{R}^{n_s}$ and $C$ the one corresponding to the control vector $c\in\mathbb{R}^{n_c}$. Both $S$ and $C$ are set as randomized dense full-rank matrices and assumed to be constant over the control horizon $T$.

The results are given in Tab.~\ref{tab:turnback} for $n_s = 12$ and $n_c = 3,6,9,12,15,18$ and $T = 10,20,30,40,50$. Due to memory limitations all examples are limited to matrices $A$ with under 35000 non-zeros of their corresponding normal form $A^TA$. 
In all cases, the recycling turnback algorithm (`Turnback-NS') manages to compute a nullspace basis in under~10~ms. In contrast, the computation times $t^*$ for the non-recycling turnback algorithm, which computes a fresh LU decomposition for each column sub-set, requires at least more than 30\% more time in all cases. For further reference, the last column indicates the computation times of solely the LU decomposition, which are about a magnitude lower. The number of column updates are given as col+ (column addition, also contains the number of columns $n$ treated in the initial LU decomposition) and col- (column deletion). While column deletions are cheaply handled by LUSOL, it does not consider sparsity of columns to be added. This could be avoided for example by developing supernodal rank-revealing LU decompositions. Still, and especially for larger problems, the computation times are comparable to the ones of the LU based NS (`LU-NS'~\eqref{eq:luns}). This is due to the higher density of the resulting nullspace bases $Z$ with a high number of non-zeros (at least upper or lower triangular instead of bands for Turnback-NS). 

We further report the density measure $d(M) = nnz(M) / (mn)$ of a matrix $M\in\mathbb{R}^{m\times n}$ as the number of non-zeros ($nnz$) by the matrix dimensions. The density of the normal form projection $Z^TZ$ demonstrates the capability of the turnback algorithm to maintain bands of the matrix $A$. For all cases the number of non-zeros in $Z^TZ$ (which is necessary for~\eqref{eq:HLSPNeNmethod}) is reduced significantly compared to $A^TA$. This is in the sense of nullspace based HLSP resolution methods where a `variable reduction', or in the case of sparse programming `non-zero reduction', is desired. At the same time for LU-NS with their triangular structure of $Z$, the resulting normal form $Z^TZ$ is fully dense ($d(Z^TZ)=1$) for all cases with a very high number of non-zeros (and oftentimes higher than the original $nnz(A^TA)$). For the case of $n_s=12$, $n_c=18$ and $T=20$, LU-NS has 68 times more non-zero entries than Turnback-NS. This stresses the importance of considering `global' sparsity patterns like bands in order to maintain computational tractability of the~\ref{eq:hlsp} sub-problem resolution.

\subsection{Solving NL-HLSP's}

\subsubsection{Hierarchy with disk, Rosenbrock and Himmelblau's constraints}
\label{eval:nlpt}

\begin{table*}[htp!]
	\centering
	\resizebox{1.\columnwidth}{!}	{%
		\begin{tabular}{@{} ccccccccccc @{}}  
			\toprule
			& & & & \multicolumn{3}{c}{Newon's method} & \multicolumn{3}{c}{Quasi-Newton method (BFGS)} \\ 			
			\cmidrule(lr){5-7}	\cmidrule(lr){8-10}	
			$l$ & & $f_l(x)\leqq v_l$ & $\Vert v_l^* \Vert_2$ & Iter. & Ref. Iter. & $\rho_{\max,l}$ & Iter. & Ref. Iter. & $\rho_{\max,l}$\\
			\midrule
			1 & Disk ineq. & $x_1^2 + x_2^2 - 1.9 \leq v_1$ & 0 & 7 & 3  & 10 & 7 & 3 & 10 \\
			2 & Ros. eq. & $(1-x_1)^2 + 100(x_2 - x_1^2)^2 = v_2$ & $2.9\cdot 10^{-4}$ & 21 & 77  & 10 &  37 & 71 & 10 \\
			3 & Disk eq. & $x_1^2 + x_2^2 - 0.9 = v_3$& 1 & 1 & 7  & 4 &  1 & 7 & 4 \\
			4 & Disk eq. & $x_2^2 + x_3^2 - 1 = v_4$& $8.3\cdot 10^{-17}$ & 1 & 7  & 10 &  1 & 7 & 8 \\
			5 & Disk ineq. & $x_4^2 + x_5^2 + 1 \leq v_5$& 1 & 1 & 13  & 4 &  1 & 13 & 10 \\
			6 & Disk eq. & $x_6^2 + x_7^2 + x_8^2 - 4 = v_6$& $1.6\cdot 10^{-10}$  & 1 & 7  & 10 &  11 & 7 & 10 \\
			7 & Ros. eq. & $(1-x_6)^2 + 100(x_7 - x_6^2)^2 = v_7$& $7.3\cdot 10^{-10}$ & 40 & 77  & 10  &  19 & 71 & 10 \\
			8 & Him. eq. &  ($x_{_9}^2 + x_{10} - 11)^2 + (x_9 + x_{10}^2 - 7)^2= v_8$& $2.1\cdot 10^{-10}$ & 1 & 21  & 10 &  7 & 21 & 10 \\
			9 & Reg. eq. & $x_{1:10} = v_9$& 4.36 & 1 & 1  & 4 &  1 & 1 & 10 \\
			\midrule
			$\Sigma$ & & & & 74 & 213 & & 85 & 201 & \\
			\bottomrule
		\end{tabular}
	}
	\caption{Non-linear test functions, s-$\mathcal{N}$IPM-HLSP: a~\ref{eq:NL-HLSP} with $p=9$ and $n=10$ composed of disk, Rosenbrock (Ros.), Himmelblau's (Him.), and regularization (Reg.) equality (eq.) and inequality (ineq.) constraints.}
	\label{tab:p8nl}
\end{table*}

\begin{figure}[htp!]
	\includegraphics[width=0.8\columnwidth]{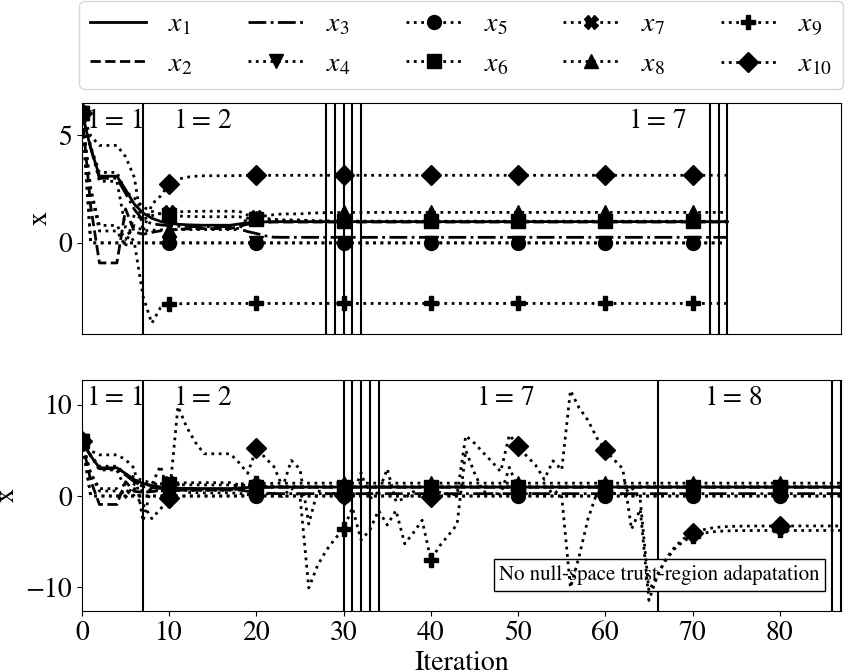}
	\centering
	\caption{Non-linear test functions, Newton's method, s-$\mathcal{N}$IPM-HLSP: Primal $x$ over S-HLSP iteration with (top) and without (bottom) nullspace trust-region adaptation. The black vertical lines indicate the current hierarchy level being resolved by the HSF.}
	\label{fig:nonlinx}
\end{figure}

\begin{figure}[htp!]
	\includegraphics[width=0.8\columnwidth]{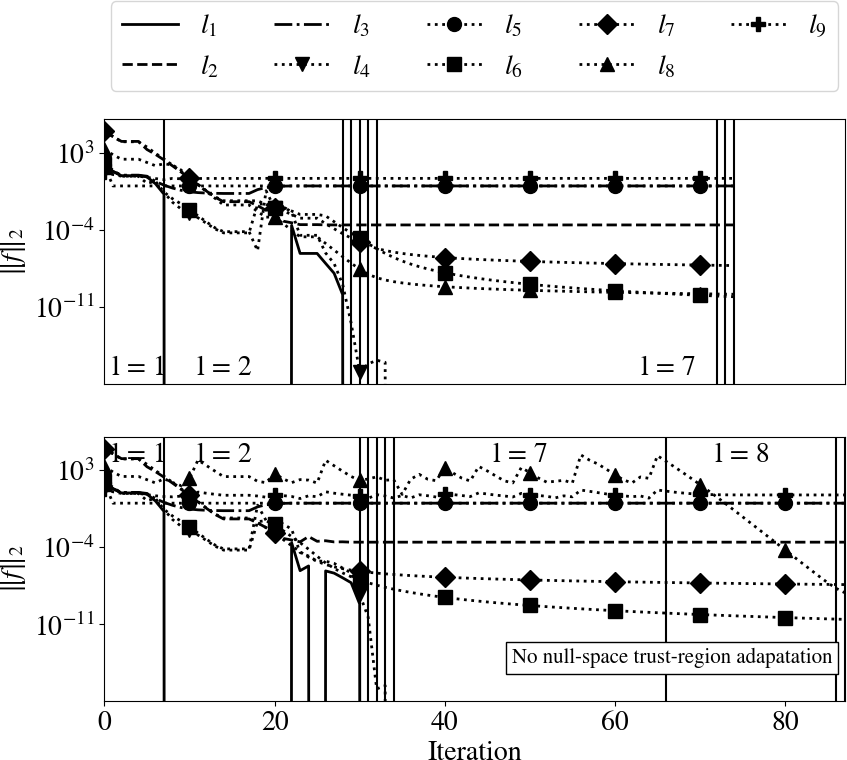}
	\centering
	\caption{Non-linear test functions, Newton's method, s-$\mathcal{N}$IPM-HLSP: error $\Vert f \Vert_2$ over S-HLSP iteration.}
	\label{fig:nonline}
\end{figure}

\begin{figure}[htp!]
	\includegraphics[width=1\columnwidth]{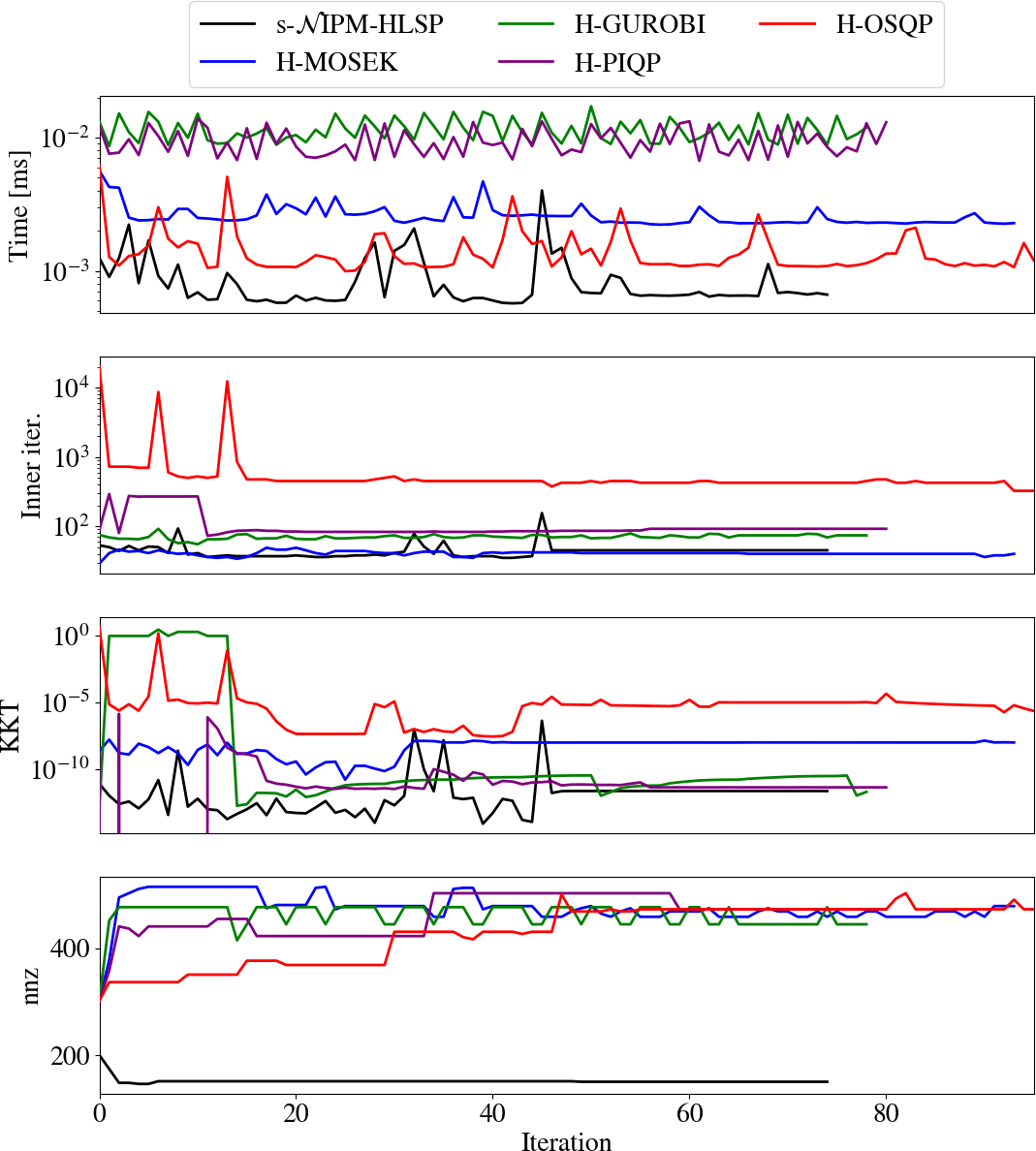}
	\centering
	\caption{Non-linear test functions, Newton's method: computation times, number of inner iterations, KKT residuals and overall number of non-zeros handled throughout the whole hierarchy  over S-HLSP outer iteration for the different HLSP sub-solvers.}
	\label{fig:nonlinsolverdata}
\end{figure}

\begin{figure}[htp!]
	\includegraphics[width=1\columnwidth]{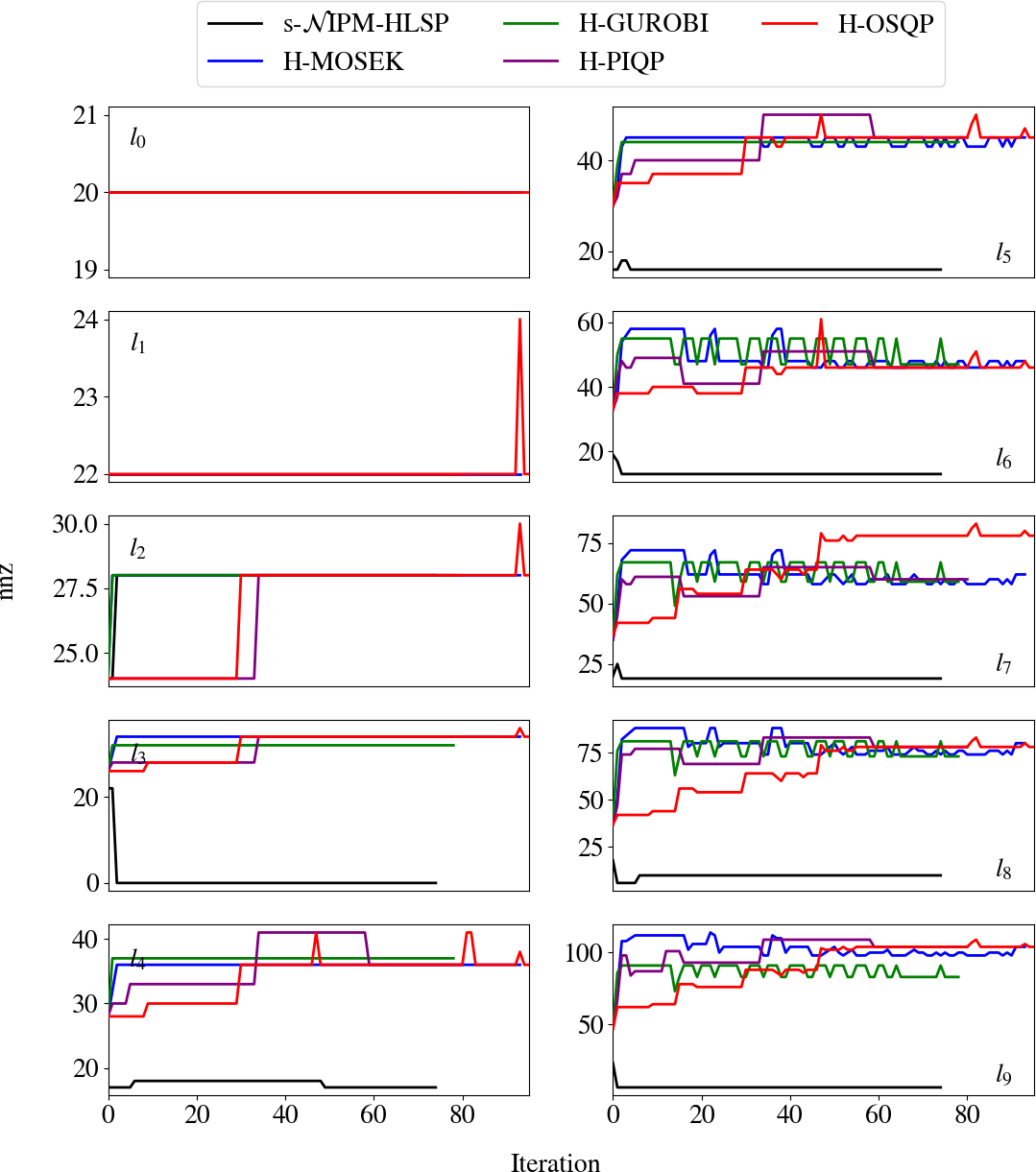}
	\centering
	\caption{Non-linear test functions, Newton's method: number of non-zeros on each priority levels for the different HLSP solvers.}
	\label{fig:nonlinnnz}
\end{figure}

We consider following~\ref{eq:NL-HLSP} given in Tab.~\ref{tab:p8nl} with $p=9$ and $n=10$. The constraints consist of quadratic disk constraints and fourth order Rosenbrock and Himmelblau's functions. 
The purpose of this hierarchy is to demonstrate capabilities of the S-HLSP in handling the following cases:
\begin{itemize}
	\item Infeasible initial point $x_0 = \BIN 6,\dots,6\BOUT$ to inequalities ($l=1, 5$) and all equalities.
	\item Algorithmic conflict between infeasible equality ($l=2$) and feasible inequality ($l=1$).
	\item Algorithmic conflict between infeasible equalities ($l= 2, 3$).
	\item Feasible solution of equality (4) in nullspace of higher priority equalities and active inequality ($l=1, 2, 3$).
	\item Infeasible inequality ($l=5$).
	\item Maintenance of optimality of singular constraint ($l= 5, 7$) during resolution of lower priority levels.
	\item Feasible solution of equality ($l=7$) in nullspace of feasible equality ($l=6$).
	\item Convergence in presence of different non-linear functions with different trust region radius requirements regarding their linear models (disk, Rosenbrock and Himmelblau's constraints).
	\item One step convergence of linear constraint ($l=9$).
\end{itemize}
Singularity refers to singularity of the Jacobians $J$ of the linearized constraints in the HLSP at a given $x$. Algorithmic conflict refers to singularity of projected constraints into the nullspace of higher priority levels while there is a conflict between the levels. It has been demonstrated in~\citep{pfeiffer2018,pfeiffer2023} that Newton's or Quasi-Newton methods are an efficient countermeasure by using second order information (and not the Gauss-Newton algorithm). The problem can be considered dense with some sparsity resulting from the variable dependency of the different functions.

The results are summarized in table~\ref{tab:p8nl} with the step threshold $\chi = 1\cdot 10^{-5}$.  The solver converges in 75 steps for Newton's method (3549 inner iterations, 0.079~s) to the minimum 
$x$ = 0.983, 0.966, 0.258, 0, 0, 1.000,
1.000, 1.414, -2.805, 3.131, see Fig.~\ref{fig:nonlinx}. For reference, we denote the necessary S-HLSP iterations to resolve each respective level in combination with the trust region constraint (ref. iter.). The disk inequality on the first level converges to a feasible point within 8 iterations. The Rosenbrock equation on the second level then converges to the infeasible point $x_{1:2} = \BIN 0.983 & 0.966\BOUT$ with $\Vert v_l^*\Vert_2 = 2.9\cdot 10^{-4}$ due to conflict with the inequality constraint on the first level. This means that the constraint is augmented with second order information. Due to the blockage of the two variables, the disk equality on level $3$ has no effect on this constraint (and converges within one iteration). Since one of the variables $x_3$ of the disk equality on level 4 is free, the constraint converges to the feasible point $x_{2:3}=\BIN 0.966 & 0.258\BOUT$. 
The disk inequality on level $5$ is infeasible but converges at the relaxed point $x_{4:5} = \BIN 0 & 0\BOUT$ with $\Vert v_5^*\Vert_2 = 1$. 
The corresponding constraint Jacobian is singular but converges cleanly with trust region radius $\rho_{\max,l} = 4$, meaning that the given second order information is accurate. On level $7$, the Rosenbrock function converges to the feasible point $x_{6:7} = \BIN 1 & 1\BOUT$ in the nullspace of the feasible point $x_{6:8} = \BIN 1 & 1 & 1.414\BOUT$ of the disk constraint on level $6$. The Himmelblau's function converges to the feasible point $x_{9:10} = \BIN -2.805 & 3.131 \BOUT$ (the feasible point $x_{9:10} = \BIN -3.779 & -283 \BOUT$ without nullspace trust region adaptation). Since at every filter step we solve the complete HLSP, the Himmelblau's function already has converged to a minimum when running the filters of the previous levels. The filter of level $8$ therefore converges within one step. The same holds for the linear regularization task on level $9$.

The error reduction over the S-HLSP iterations is given in Fig.~\ref{fig:nonline}. If the nullspace trust-region adaptation is not applied, it can be observed that the two variables $x_9$ and $x_{10}$ corresponding to the Himmelblau's function on level 8 behave highly erratic, see Fig.~\ref{fig:nonlinx}. This is due to the fact that the trust-region radii applicable during the resolution of the HSF levels 1 to 8 (Rosenbrock and disk constraint) are not appropriate for the Himmelblau's function. Therefore, convergence takes 87 iterations. This is in contrast to the HSF with nullspace trust region adaptation. Here,  during the resolution of the higher priority levels the free variables of the Himmelblau's function are already appropriately adjusted by the trust region adaptation such that the Himmelblau's task error is reduced simultaneously. This way, an acceptable solution is approximately obtained after the resolution of level 2 at S-HLSP iteration 28. This is in contrast to the case without adaptation, only leading to convergence once the HSF actually resolves level 8 containing Himmelblau's function from iteration 66 onwards.

The detailed solver times of s-$\mathcal{N}$IPM-HLSP per S-HLSP outer iteration are given in Fig.~\ref{fig:nonlinsolverdata}. As can be seen, they are about three times lower than the ones of the next best IPM solver H-MOSEK with a comparable number of inner iterations. This can be explained by the significantly lower number of non-zeros that need to be handled by our solver based on the reduced Hessian formulation. This offsets the additional computation time which is required to compute a nullspace basis of the active constraints via the turnback algorithm and to project the remaining constraints. H-OSQP is the second-fastest solver. However, the non-linear constraint residuals are significantly higher than the ones seen for the IPM based solvers. For example, the Himmelblau function on level 8 is only resolved to $1.5\cdot 10^{-5}$ (s-$\mathcal{N}$IPM-HLSP: $2.1\cdot 10^{-10}$). This is also reflected in the comparatively high KKT residual of H-OSQP.

Figure~\ref{fig:nonlinnnz} shows the number of non-zeros of the constraint matrices of each respective level. On the second level $l_2$ (Rosenbrock function), it can be observed that the number of non-zeros increases in the second outer iteration. This is due to the switch form the Gauss-Newton algorithm to the Newton's method (`Hessian augmentation'). As the Hessian is full-rank on the corresponding variables, the null-space projection of the subsequent level $l_3$ (disk constraint) vanishes entirely as the corresponding constraint occupies the same variables. s-$\mathcal{N}$IPM-HLSP does not resolve levels with empty constraint matrices $A_{\mathbb{E}_l}$ and  $A_{\mathbb{I}_l}$ and therefore $nnz=0$ of level $l_3$. Note that this assumes that the solver has successfully identified feasible ($v^*=0$, or optimal infeasible $v^*\neq0$) points of the previous priority levels. Similarly, in subsequent levels the change of number of non-zeros indicates a switch between the hierarchical Gauss-Newton algorithm and the Newton's method. 

It can be observed for all solvers except s-$\mathcal{N}$IPM-HLSP that the number of non-zeros increases with the progression through the hierarchy. This is due to the growth of number of constraints in $A_{\mA_{l-1}}$, which contains all constraints of the previous levels $1$ to $l-1$. In contrast, s-$\mathcal{N}$IPM-HLSP does not need to consider these constraints to the expense of computing appropriate nullspace bases such that $A_{\mA_{l-1}}N_{l-1}=0$. 

\subsubsection{Inverse kinematics}
\label{eval:ik}

\begin{table}[htp!]
	\centering
	\begin{tabular}{@{} ccccccc @{}}  
		\toprule
		& & &\multicolumn{2}{c}{Newon's method} & \multicolumn{2}{c}{BFGS} \\ 			
		\cmidrule(lr){4-5}	\cmidrule(lr){6-7}	
		$l$ & $f_l(x) \leqq v_l$ & $\Vert v_l^* \Vert_2$ & Iter. & $\rho_{\max,l}$ & Iter. & $\rho_{\max,l}$\\
		\midrule
		1 & J. lim. ineq. & 0 & 1 & 10 &  1 & 10 \\
		2 & LF, RF, LH eq. & $4\cdot 10^{-8}$ & 5  & 10 &  10 & 10 \\
		3 & CoM ineq. & 0 & 20  & $3.1\cdot 10^{-4}$ &  14 &  $6\cdot10^{-4}$ \\
		4 & Right hand eq. & 1.13 & 4  & $1.3\cdot 10^{-1}$ &  5  & $3.2\cdot10^{-2}$ \\
		5 & Reg. eq. & 6.73 & 0 & 10 &  0 &  10 \\
		\midrule
		$\Sigma$ & & & 30 & & 30 \\
		\bottomrule
	\end{tabular}
	\caption{HRP-2, s-$\mathcal{N}$IPM-HLSP: a~\ref{eq:NL-HLSP} with $p=5$ and $n=38$ for the humanoid robot HRP-2. J. lim.: Joint limits, LF: left foot, RF: right foot, LH: left hand.}
	\label{tab:HRP-2}
\end{table}

\begin{figure}[htp!]
	\includegraphics[width=0.5\columnwidth]{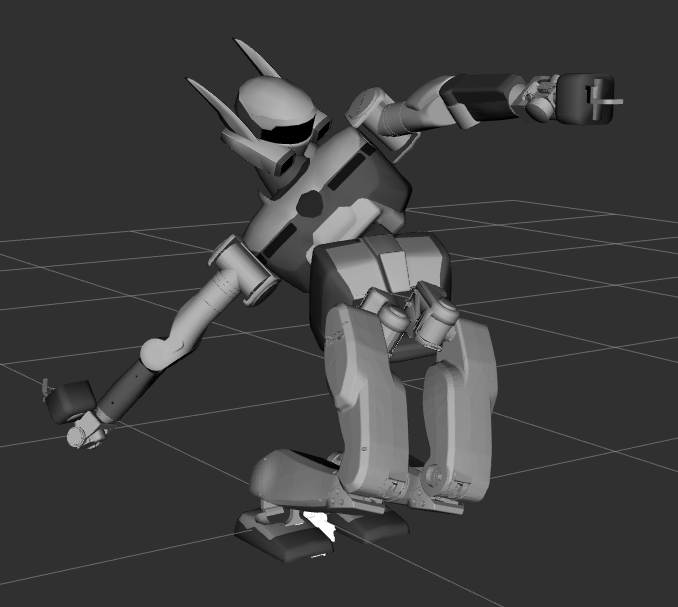}
	\centering
	\caption{HRP-2, Newton's method, s-$\mathcal{N}$IPM-HLSP: converged posture of HRP-2 reaching for a target far below its feet.}
	\label{fig:HRP-2_stretch}
\end{figure}

\begin{figure}[htp!]
	\includegraphics[width=0.8\columnwidth]{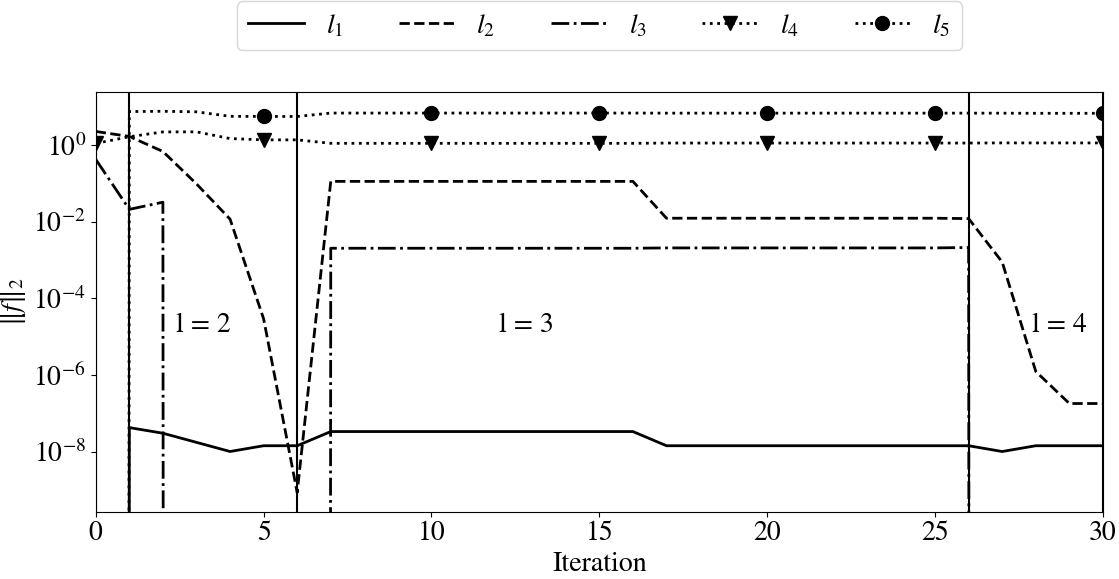}
	\centering
	\caption{HRP-2, Newton's method, s-$\mathcal{N}$IPM-HLSP: task error over S-HLSP iteration for Newton's method.}
	\label{fig:HRP-2e}
\end{figure}

\begin{figure}[htp!]
	\includegraphics[width=0.8\columnwidth]{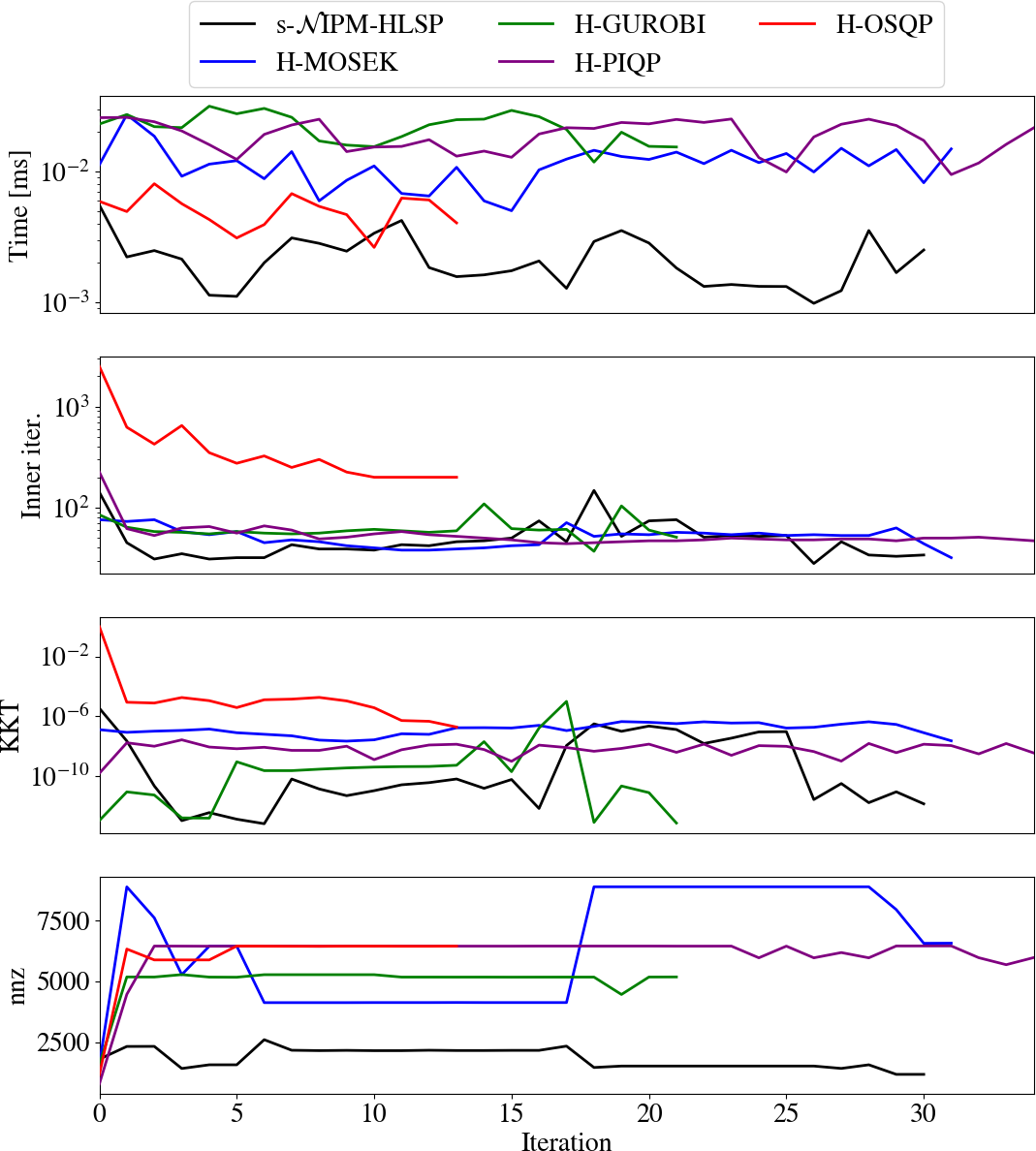}
	\centering
	\caption{HRP-2, Newton's method: solver times, number of inner iterations, KKT residuals and overall number of non-zeros handled throughout the whole hierarchy  over S-HLSP outer iteration for the different HLSP sub-solvers.}
	\label{fig:HRP2solverdata}
\end{figure}

In this example, we solve an inverse kinematics problem for the humanoid robot HRP-2 with $n=38$ degrees of freedom. We set $x_0=0$. The kinematics of the robot are described by series of trigonometric functions except for the base of the robot (corresponding to the robot's lower torso), which enables linear translations and rotations. The robot hierarchy with $p=5$ is given in Tab.~\ref{tab:HRP-2}.
The first level contains the joints' lower and upper limits. The second level prescribes the contact positions of the left and right feet and the left hand. The third level restricts the CoM position to a bounding box. The fifth level lets the right hand of the robot reach for an out-of-reach target $\BIN -0.5 & -0.5 & -1\BOUT$~m below its feet. For reference, the right foot is at $\BIN 0.015 &  -0.1 & 0.1\BOUT$~m with the $z$ component approximately at ground level $0.1$~m. On the last level all variables are regularized to zero.
This problem is typically dense with some sparse structure resulting from kinematic chains or bound constraints.

The converged posture of the robot is depicted in Fig.~\ref{fig:HRP-2_stretch}. The problem is solved in 30 outer iterations and 1589 inner iterations in 0.11~s (30 outer iterations, 983 inner iterations, 0.07~s for BFGS). 
A high accuracy solution is obtained for Newton's method while the results for BFGS are slightly worse (right hand converges to $\Vert v_4^* \Vert_2 = 1.41$~m for BFGS and $\Vert v_4^* \Vert_2 = 1.13$~m for Newton's method).

The task error over the S-HSLP iterations is depicted in Fig.~\ref{fig:HRP-2e}. It can be observed that the task error of level 2 increases during the resolution of level 3. This is due to the relaxed step threshold of $\chi = 1\cdot10^{-3}$. However, convergence with high accuracy is achieved eventually after the resolution of level 4.

It is apparent that the resulting posture in Fig.~\ref{fig:HRP-2_stretch} is not optimal as the left arm could be bent further downwards. This is due to the Hessian augmentation of the left hand task which prevents the resolution of the lower priority right arm stretch task in its nullspace of the linearized constraint. A similar issue can be observed for the right wrist joint. For this problem we chose the augmentation threshold $\epsilon = 10^{-9}$. 

Our proposed sparse HLSP solver s-$\mathcal{N}$IPM-HLSP solves the HLSP sub-problems the fastest as can be seen from Fig.~\ref{fig:HRP2solverdata}. All IPM based solvers require approximately the same number of inner iterations while delivering high accuracy solutions with low primal and dual residuals. However, the number of non-zeros is about 50\% lower for the reduced Hessian formulation (s-$\mathcal{N}$IPM-HLSP), which explains the low computation times of our solver. H-OSQP converges after fewer outer iterations, but to a less optimal solution ($\Vert v_4^* \Vert_2 = 1.73$~m).

\subsubsection{Prioritized discrete non-linear optimal control}
\label{eval:trajopt}

\begin{table}[htp!]
	\centering
	\setlength{\tabcolsep}{1pt}
	\begin{tabular}{@{} ccccc @{}}  
		\toprule
		$l$ & $f_l(x) \leqq v_l$& $\Vert v_l^* \Vert_2$ & Iteration & $\rho_{\max,l}$ \\
		\midrule
		& with $t = 0,\dots,T-1$\\
		1 & 
		$\BIN -c_{t+1} + c_{\min} \\
		\vert F_{x,y,{t+1}}^i \vert -  F_{x,y,{t+1}\max}\\
		-F_{z,{t+1}}^i \\
		F_{z,{t+1}}^i - F_{z,{t+1},\max}
		\BOUT
		\leq v_{1,t+1}$ & $2.6\cdot 10^{-8}$ & 1 & 5\\ 
		& $i = 1,\dots,4$ \\
		2 & $\Vert c_{t+1} - r_{i,{t+1}}\Vert^2_2 - L^2 \leq  v_{2,{t+1}}$ & 0 & 4  & 10 \\
		3 & 
		$\BIN c_{t+1} - c_t - \dot{c}_t\Delta t\\
		\dot{c}_{t+1} - \dot{c}_t - \sum_{i=1}^{4}\delta_t^i F_{t}^i\Delta t/m - g\Delta t\\
		k_{t+1} - k_t - \sum_{i=1}^{4} \delta_t^i(r_t^i - c_t)\times F_t^j\BOUT= v_{3,{t+1}}$  & $2.9\cdot10^{-13}$ & 8  & 10 \\
		4 & $c_{t+1} - c_{d,{t+1}} = v_{4,{t+1}}$ & 0.305 & 0  & 10 \\
		5 & $I\omega_{t+1}^2 = v_{5,{t+1}}$ & 0.29 & 1 & 10\\
		6 & $\BIN c^T_{t+1} & \dot{c}^T_{t+1} & \omega_{t+1}^T & F_{t+1}^T \BOUT^T = v_{6,{t+1}}$ & 149.715 & 0 & 10\\
		\midrule
		$\Sigma$ & & & 14\\
		\bottomrule
	\end{tabular}
	\caption{Solo12, Newton's method, s-$\mathcal{N}$IPM-HLSP: a~\ref{eq:NL-HLSP} with $p=6$ and $n=525$ for the robot dog Solo12.}
	\label{tab:solo12}
\end{table}

\begin{figure}[htp!]
	\includegraphics[width=1\columnwidth]{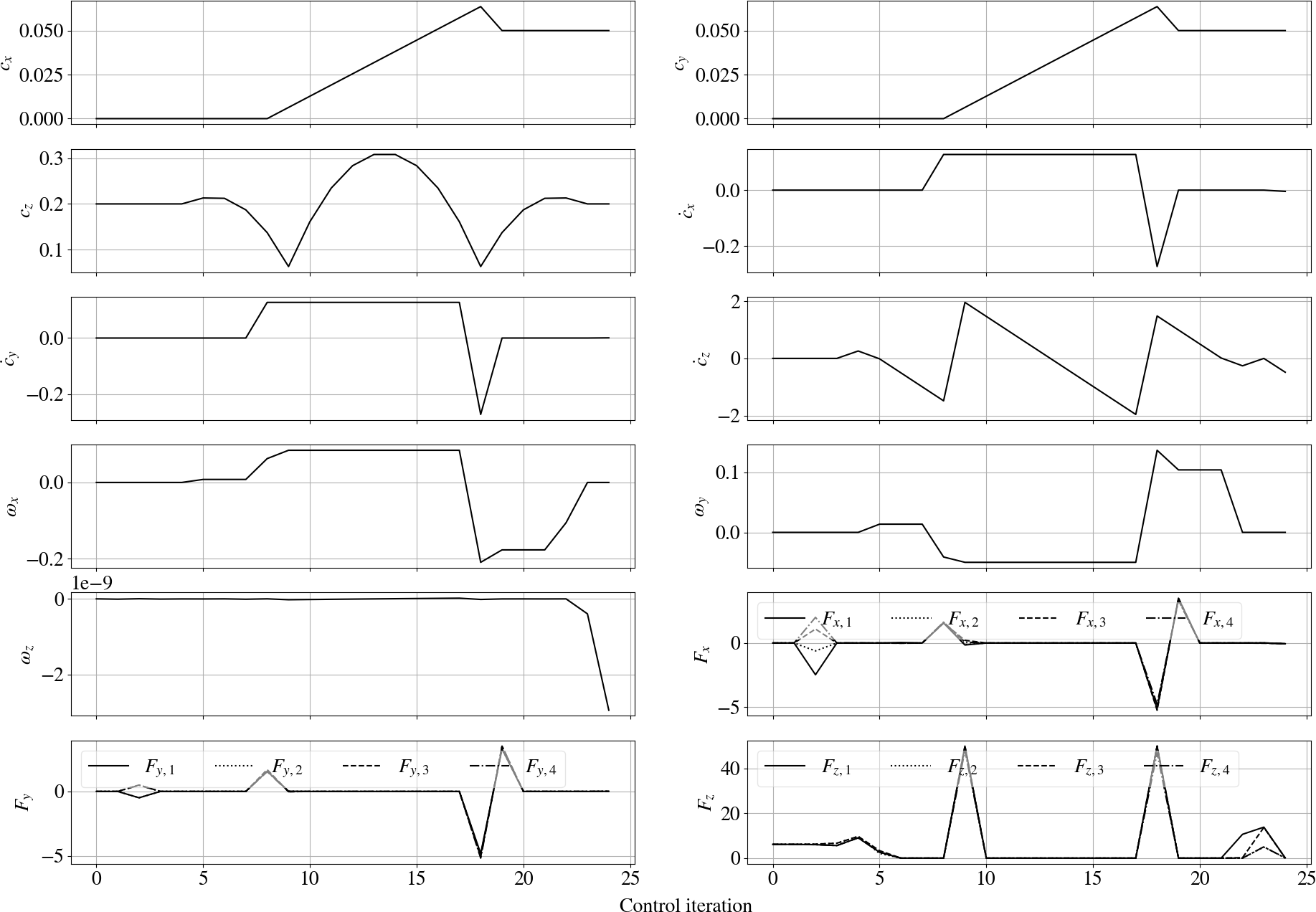}
	\centering
	\caption{Solo12, Newton's method, s-$\mathcal{N}$IPM-HLSP: CoM $c$, CoM linear and angular velocity $\dot{c}, \omega$ and contact forces $F_x$, $F_y$ and $F_z$ for the four end-effectors.}
	\label{fig:solo12rwf}
\end{figure}

\begin{figure}[htp!]
	\includegraphics[width=0.8\columnwidth]{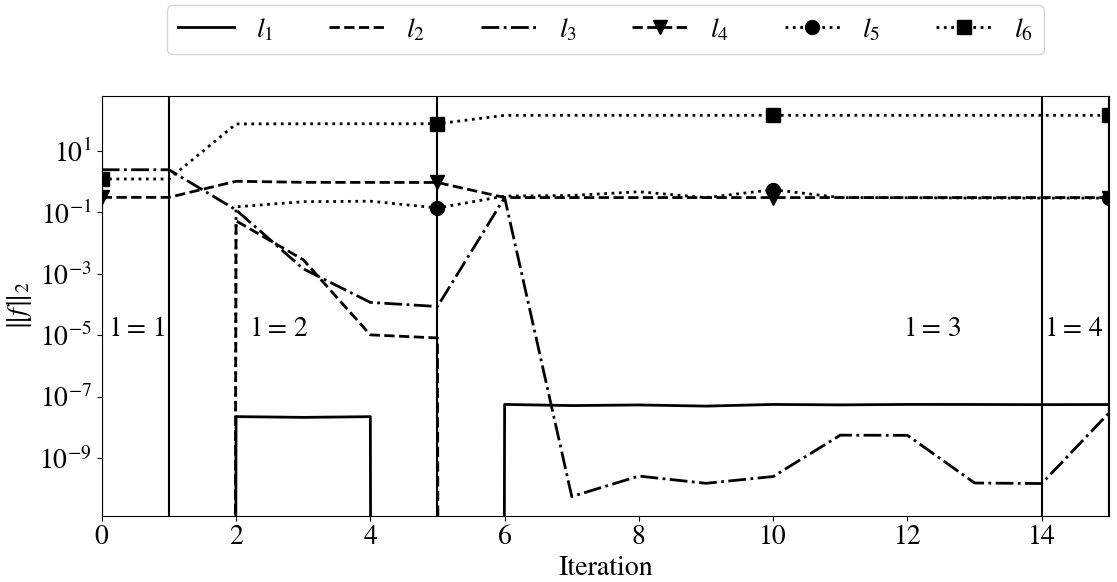}
	\centering
	\caption{Solo12, Newton's method, s-$\mathcal{N}$IPM-HLSP: reduction of task error for each level over S-HLSP iteration.}
	\label{fig:solo12e}
\end{figure}

\begin{figure}[htp!]
	\includegraphics[width=0.8\columnwidth]{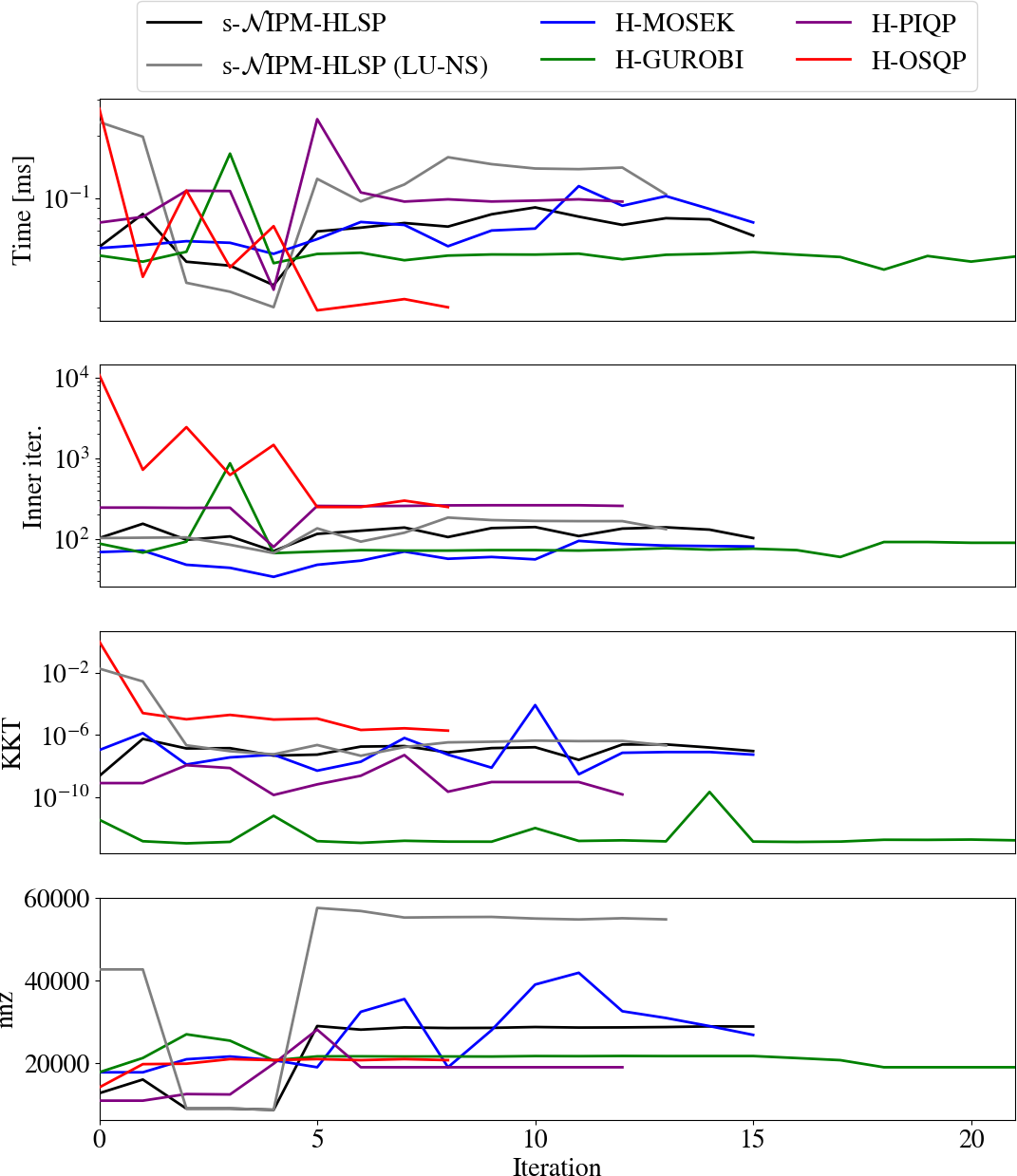}
	\centering
	\caption{Solo12, Newton's method: computation times, number of inner iterations, KKT residuals and overall number of non-zeros handled throughout the whole hierarchy  over S-HLSP outer iteration for the different HLSP sub-solvers.}
	\label{fig:solo12solverdata}
\end{figure}

\begin{figure}[htp!]
	\includegraphics[width=0.8\columnwidth]{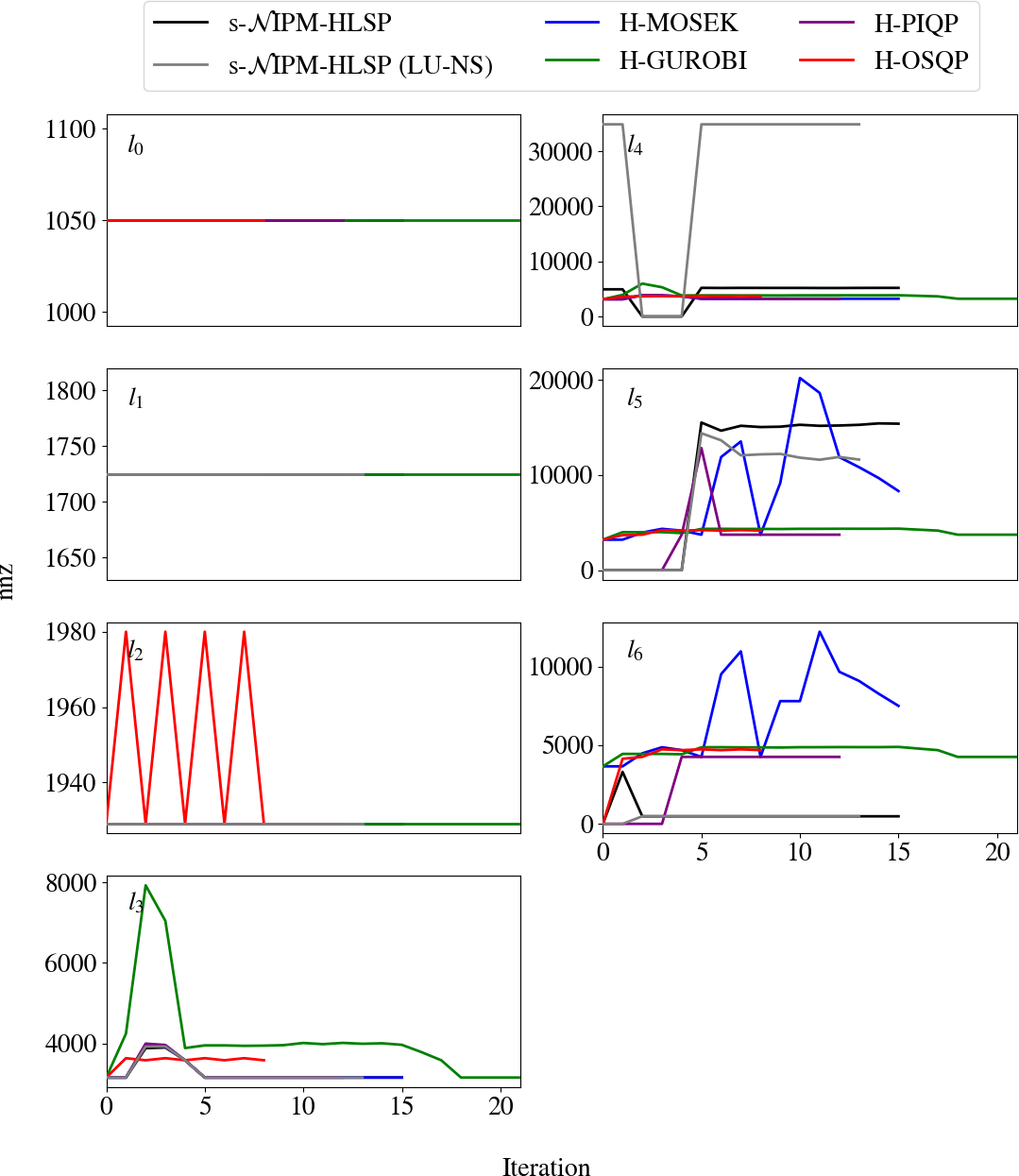}
	\centering
	\caption{Solo12, Newton's method: Number of non-zeros per priority level $l$.}
	\label{fig:solo12nnz}
\end{figure}

In this example, we solve a prioritized non-linear optimal control problem discretized by the explicit Euler method over a horizon of $T=25$ with $\Delta t = 0.05$~s. Specifically, we want to compute a CoM trajectory of the robot dog Solo12 ($m=2.5$~kg, moment of inertia $I=\diag(0.03,0.051,0.067)$~kg/m$^2$) 	corresponding to a diagonal forward leap. 
The dynamics of the robot are described by the centroidal dynamics at the CoM, see~\citep{meduri2022} and level 3 of the control hierarchy in Tab.~\ref{tab:solo12}. The robot has four legs, which can exert forces on the ground in $x$, $y$ and $z$ direction up to a given threshold $F_{\max}$. Thereby, only positive forces in normal / $z$ direction to the ground are allowed (only pushing, no pulling). Consequently, we have 12 control variables per time step $n_u=12$. Together with the CoM position $c$ and linear and angular velocities $\dot{c}$ and $\omega$ ($n_x=9$), this amounts to $n=T(n_x + n_u) = 525$ variables. The robot contacts are indicated in the centroidal dynamics by the binary operator $\delta_t^i$ (1 for contact at time $t$ of foot $i$, otherwise 0). 
The four feet are in contact with the ground until control iteration 10 at $r_1 = \BIN 0.2 & 0.142 & 0.015\BOUT$~m, $r_2 = \BIN 0.2 & -0.142 & 0.015\BOUT$~m, $r_3 = \BIN -0.2 & 0.142 & 0.015\BOUT$~m, $r_4 = \BIN -0.2 & -0.142 & 0.015\BOUT$~m, and then from control iteration 18 to 25 at $r_1 = \BIN 0.25 & 0.182 & 0.015\BOUT$~m, $r_2 = \BIN 0.25 & -0.102 & 0.015\BOUT$~m, $r_3 = \BIN -0.15 & 0.182 & 0.015\BOUT$~m, $r_4 = \BIN -0.15 & -0.102 & 0.015\BOUT$~m. Between control iterations 10 and 18 the feet are prescribed to not be in contact with the ground. 
The desired CoM position is set to $c_{d} = \BIN 0 & 0 & 0.2\BOUT$~m until control iteration 12 and $c_{d} = \BIN 0.05 & 0.05 & 0.2\BOUT$~m afterwards ($l=4$). On level 5, we wish to minimize the rotational energy $\omega^2$ of the robot. On the last level, all states and controls are regularized to zero. We set $x_0$ to zero except for the CoM $c$ which is set to the desired one $c_d$. The problem is in the form of discrete optimal control problems~\eqref{eq:oc} with their typical block-diagonal structure~\eqref{eq:blockA}, which can be leveraged by the turnback algorithm.

A robot trajectory is identified in 14 iterations by Newton's method and s-$\mathcal{N}$IPM-HLSP. The overall computation time for the~\ref{eq:hlsp} sub-problems is 1.01~s with 1475 inner iterations. The step threshold is chosen as $\chi=0.1$. We use block-diagonal finite differences to compute the hierarchical Hessians. The CoM trajectory is given in the top three graphs of Fig~\ref{fig:solo12rwf}. As expected, the robot's CoM is moved up and diagonally forward. $c_x$ and $c_y$ increase linearly after an impulsive force $F_x$ and $F_y$ in $x$ and $y$ direction at control iteration 8. To account for the contact switch, the robot's CoM is moved upwards (up to $c_z=0.3$~m) after a force impulse of 50~N (within the limit $F_{z,\max} = 50$~N; $F_{x,y,\max} = 20$~N). The effect of the angular velocity regularization is clearly recognizable ($l=5$), with zero angular velocities $\omega_x$, $\omega_y$ and $\omega_z$ at the end of the movement. Similarly, the contact forces are zero whenever they are not necessary due to the regularization task on the last level. For example, the $F_z$ force components are zero during the jump from control iteration 10 to 17.

Figure~\ref{fig:solo12solverdata} depicts the computation times of the different HLSP solvers. The computation time for H-MOSEK is 1.19~s with 1040 inner iterations with a solution similar to the one of s-$\mathcal{N}$IPM-HLSP (1.01~s with 1475 inner iterations). While there is a slight edge in computational efficiency for s-$\mathcal{N}$IPM-HLSP, we believe that our solver can benefit from further computational improvements, especially with regards to the computation of the nullspace basis via turnback algorithm (for example sparse LU column updates). 

The difference in the number of non-zeros of the reduced normal form is not as pronounced as in the previous examples. As can be seen in Fig.~\ref{fig:solo12nnz}, the number of non-zeros on the forth level is higher for s-$\mathcal{N}$IPM-HLSP (5200 non-zeros) than for H-MOSEK (3200 non-zeros). This can be explained by the fact that a high number of sparse bound constraints (trust-region and force limits) are projected into the nullspace of the robot dynamics of level 3 (which consists of dense(r) blocks). Nonetheless, compared with the number of non-zeros of s-$\mathcal{N}$IPM-HLSP (LU-NS) based on the dense nullspace basis~\eqref{eq:luns}, the band preserving quality of the turnback algorithm can be clearly recognized on level 4. Additionally, the effect of reduced Hessian methods in hierarchical programming becomes more apparent with higher number of priority levels $p$. On level 6, the number of non-zeros for s-$\mathcal{N}$IPM-HLSP is 500 while H-MOSEK needs to handle above 4000 non-zeros in most iterations.

The solutions for H-PIQP, H-GUROBI and H-OSQP are less optimal. For example, H-GUROBI resolves level 4 only to error $\Vert v_4^*\Vert_2 = 0.95$ as compared to $\Vert v_4^*\Vert_2 = 0.31$ for s-$\mathcal{N}$IPM-HLSP and H-MOSEK. Note that S-HLSP is not a global optimizer and convergence depends on the quality of the resolution of the HLSP sub-problems (a global solution to the HLSP sub-problem needs to be identified, see~\cite{fletcher2002b}). Also, as can be seen from the oscillating number of non-zeros on level 2 of H-OSQP (see Fig.~\ref{fig:solo12nnz}), tuning of the Hessian augmentation threshold can turn out difficult and motivates the design of more advanced switching methods between the Gauss-Newton algorithm and Newton's method.

\section{Conclusion}
With this work we have conceptualized, implemented and demonstrated a sequential hierarchical least-squares programming solver with trust region and hierarchical step-filter for non-linear hierarchical least-squares programming. The sparsity of prioritized discrete optimal control is thereby efficiently exploited by virtue of an efficient implementation of the turnback algorithm. The solver is characterized by global convergence to a local KKT point, which was confirmed in a set of evaluations.

In future work we would like to bring the solver to higher levels of maturity, for example by designing a rank revealing supernodal LU decomposition~\citep{demmel1995} with column updating capabilities.
We also see the necessity to develop test benches tailored to~\ref{eq:NL-HLSP} for more comprehensive testing, similarly to the ones for NLP like CUTEr~\citep{Bongartz1995}. 

\section*{Funding}
This work is partly supported by the Schaeffler Hub for Advanced Research at Nanyang Technological University, under the ASTAR IAF-ICP Programme ICP1900093.
This work is partly supported by the Research Project I.AM. through the European Union H2020 program (GA 871899).

\bibliographystyle{tfcad}
\bibliography{bib}



\listoffigures

\end{document}